\documentclass{article}
\usepackage{fullpage}
\usepackage{comment}
\usepackage{mathtools}%
\usepackage{amsfonts}%
\usepackage{amssymb}%
\usepackage{amsthm}

\usepackage{graphicx}
\usepackage{xspace}
\usepackage{url}
\usepackage[english]{babel}

\usepackage{cleveref}
\usepackage{float}
\usepackage[all]{xy}
\usepackage{etoolbox}
\usepackage{latexsym}
\usepackage{enumerate}
\usepackage{cite}

\usepackage{color}

\newcommand{\field}[1]{\mathbb{#1}}
\newcommand{\inpr}[2]{\langle #1,#2 \rangle}

\newcommand{\N}{\field{N}}

\newcommand{\R}{\field{R}}
\newcommand{\C}{\field{C}}
\renewcommand{\H}{\field{H}}
\newcommand{\D}{\field{D}}

\renewcommand{\Re}{\mathrm{Re}}

\newcommand{\lhol}{\mathcal{SH}_L}
\newcommand{\rhol}{\mathcal{SH}_R}
\newcommand{\intrin}{\mathcal{N}}
\newcommand{\sderiv}[1][]{\partial_{S#1}}
\newcommand{\rsderiv}[1][]{\partial_{\overset{\leftarrow}{S#1}}}
\newcommand{\bpartial}{\overline{\partial}}
\newcommand{\rpartial}[1]{\partial_{\overset{\leftarrow}{#1}}}
\newcommand{\rbpartial}[1]{\bpartial_{\overset{\leftarrow}{#1}}}
\newcommand{\berg}{\mathcal{A}}

\newcommand{\holder}{\frac{1}{p}+\frac{1}{q} = 1}
\newcommand{\res}[1]{\operatorname{Res}_{#1}}
\newcommand{\Tr}[1]{\operatorname{Tr}_{#1}}
\newcommand{\tr}{\operatorname{tr}}
\newcommand{\lift}[1]{\operatorname{Lift}_{#1}}
\newcommand{\B}{\mathcal{B}}
\newcommand{\boundOP}{\B}
\newcommand{\hil}{\mathcal{H}}
\newcommand{\ran}{\mathrm{ran}}
\newcommand{\linspan}{\operatorname{span}}
\newcommand{\bi}{\boldsymbol{i}}
\renewcommand{\S}{\mathbb{S}}
\newcommand{\ext}{\operatorname{ext}}

\newcommand{\id}{\mathcal{I}}

\theoremstyle{plain}
\newtheorem{theorem}{Theorem}[section]
\newtheorem{lemma}[theorem]{Lemma}
\newtheorem{proposition}[theorem]{Proposition}

\newtheorem{corollary}[theorem]{Corollary}

\theoremstyle{definition}
\newtheorem{definition}[theorem]{Definition}
\newtheorem{remark}[theorem]{Remark}

\date{}
\title{\bf Schatten class and Berezin transform of quaternionic linear operators}

\author{
Fabrizio Colombo\\
Politecnico di Milano\\
Dipartimento di Matematica\\
Via E. Bonardi, 9\\
20133 Milano, Italy\\
fabrizio.colombo@polimi.it\\
\and
Jonathan Gantner \\
Politecnico di Milano\\
Dipartimento di Matematica\\
Via E. Bonardi, 9\\
20133 Milano, Italy\\
jonathan.gantner@polimi.it\\
\and
Tim Janssens\\
University of Antwerp\\
Department of Mathematics\\
Middelheimlaan, 1\\
2020 Antwerp, Belgium\\
tim.janssens4@uantwerpen.be
}

\begin{document}

\maketitle

\begin{abstract}
In this paper we introduce the Schatten class of operators
 and the Berezin transform of operators in the quaternionic setting. The first topic is of great importance in operator theory but  it is also
  necessary to study the second
 one because we need the notion of trace class operators, which is a particular case of the Schatten class.
 Regarding the Berezin transform, we give  the general definition and properties. Then we concentrate
 on the  setting of weighted Bergman spaces of slice hyperholomorphic functions.
 Our results are based on the $S$-spectrum for quaternionic operators, which is the notion of spectrum that appears
 in the quaternionic version of the spectral theorem and in the quaternionic $S$-functional calculus.
\end{abstract}
\vskip 1cm
\noindent
\section{Introduction}

The study of quaternionic linear operators was stimulated by  the celebrated paper  of
 G. Birkhoff and  J. von Neumann \cite{BvN} on the logic of quantum mechanics,  where they proved that there are essentially
 two possible ways to formulate quantum mechanics: the well known one using complex numbers and also one using quaternions.
 The development of quaternionic quantum mechanics
 has been done by several authors in the past. Without claiming completeness, we mention: S. Adler, see \cite{adler},
  L. P. Horwitz and L. C. Biedenharn, see \cite{21},
D. Finkelstein, J. M. Jauch, S. Schiminovich and  D. Speiser, see \cite{14}, and  G. Emch see \cite{12}.
The notion of spectrum of a quaternionic linear operator used in the past was the right spectrum. Even though this notion
of spectrum gives just a partial description of the spectral analysis of a quaternionic operator, basically
 it contains just the eigenvalues,
several results in quaternionic quantum mechanics have been obtained.

 Only in 2007, one of the authors introduced with some collaborators the notion of $S$-spectrum of a quaternionic operator, see the book \cite{SCalcBook}.
 This spectrum was suggested by the quaternionic version of the Riesz-Dunford functional calculus
 (see \cite{ds,rudin} for the classical results), which is called $S$-functional calculus or quaternionic functional calculus and whose full development was done in  \cite{acgs,JGA,CLOSED}.
 Using the notion of right spectrum of a quaternionic matrix, the spectral theorem was proved in  \cite{fp}
 (see \cite{ds2} for the classical result).
 The spectral theorem for quaternionic bounded or unbounded normal operators based on the $S$-spectrum
  has been proved recently in the papers \cite{ack,acks2}. The case of compact operators is in \cite{spectcomp}.

  The  literature contains several attempts to prove the quaternionic version of the spectral theorem, but the notion of spectrum in the general case was not made clear, see \cite{sc,Viswanath}.
  The fact that the $S$-eigenvalues and the right eigenvalues are equal explains why some results in quaternionic quantum mechanic
  were possible to be obtained without the notion of $S$-spectrum.
  More recently, the use of the $S$-spectrum can be found in \cite{Thirulogasanthar}.
  \\

 The aim of this paper is to introduce the Schatten class of quaternionic operators and the Berezin transform.
 In the following, a complex Hilbert space will be denoted by $\hil_\C$ while we keep the symbol $\hil$ for a quaternionic Hilbert space.
We recall some facts of the classical theory in order to better explain the quaternionic setting.
Suppose that $B$ is a compact self-adjoint operator on a complex separable Hilbert space $\mathcal{H}_{\mathbb{C}}$. Then it is well known that
there exists a sequence of real numbers $\{\lambda_n\}$ tending to zero and there exists an orthonormal set
$\{u_n\}$ in $\hil_\C$ such that
$$
Bx=\sum_{n=1}^\infty \lambda_n\langle x,u_n\rangle u_n,\ \ \ \ {\rm for\ all}\ \ x\in \hil_\C.
$$
When $A$ is a compact operator, but not necessarily self-adjoint, we consider the polar decomposition
$A=U|A|$, where $|A|=\sqrt{A^*A}$ is a positive operator and $U$ is a partial isometry.
From the spectral representation of compact self-adjoint operators and the polar decomposition of bounded linear operators, we
get the representation of the compact operator $A$ as
$$
Ax=\sum_{n=1}^\infty \lambda_n\langle x,u_n\rangle v_n,\ \ \ \ {\rm for\ all}\ \ x\in \hil_\C,
$$
where $v_n:=Uu_n$ is an orthonormal set.
The positive real numbers $\{\lambda_n\}_{n\in \mathbb{N}}$ are called singular values of $A$.

Associated with the singular values of $A$, we have the Schatten class  operators $S_p$ that consists of those
operators for which the sequence $\{\lambda_n\}_{n\in \mathbb{N}}$ belongs to $\ell^p$ for $p\in (0,\infty)$.
In the case $p\in [1,\infty)$ the sets $S_p$ are Banach spaces.
\\
Two important spaces are $p=1$, the so called trace class, and $p=2$, the Hilbert-Smith operators.
\\
The Schatten class  operators were apparently introduced to study integral equations with non
symmetric kernels as pointed out in the book of I. C. Gohberg and M. G. Krein \cite{GK}, later on it was discovered
that the properties of the singular values play an important role in the construction of the general
theory of symmetrically normed ideals of compact operators.
\\
\\
Let us point out that in the quaternionic setting, it has been understood just recently
what the correct notion of spectrum is that replaces the classical
notion of spectrum of operators on a complex Banach space. This spectrum is called $S$-spectrum and
its is also defined for $n$-tuples of noncommuting operators, see \cite{SCalcBook}.

Precisely,  let $V$ be a bilateral quaternionic
Banach space and denote by $\boundOP(V)$ the Banach space of all quaternionic bounded linear operators on $V$
endowed with the natural norm.
Let $T: V\to V$ be a quaternionic bounded linear operator (left or right linear). We define the $S$-spectrum of $T$ as
$$
\sigma_S(T)=\{s\in \mathbb{H}\ :\ T^2-2\Re(s)T +|s|^2\mathcal{I}\ \ {\rm is\ not\ invertible\ in }\  \ \mathcal{B}(V)\}
$$
where $s=s_0+s_1e_1+s_2e_2+s_2e_3$ is a quaternion and $e_1$, $e_2$ and $e_3$ is the standard basis of the
 quaternions $\mathbb{H}$,
$\Re(s)=s_0$ is the real part  and $|s|^2$ is the squared Euclidean norm.
The $S$-resolvent set is defined as
$$
\rho_S(T)=\mathbb{H}\setminus \sigma_S(T).
$$

The notion of $S$-spectrum for quaternionic operators arises naturally in the slice
hyperholomorphic functional calculus, called $S$-functional calculus or quaternionic functional calculus,
which is the quaternionic analogue
of the Riesz-Dunford functional calculus for complex operators on a complex Banach space.
Recently, it turned out that also the spectral theorem for quaternionic operators
(bounded or unbounded) is based on the $S$-spectrum.

This fact restores for the quaternionic setting the well known fact
that  both the Riesz-Dunford functional calculus
and the spectral theorem are based on the same notion of spectrum.

This  was not clear for a long time because the literature related to quaternionic linear operators used the notion of right spectrum of a quaternionic operator, which,
for a right linear operator $T$, is defined as
$$
\sigma_R(T)=\{s\in \mathbb{H}\ \ such\ \ that  \ \exists  \ \  0\not=v\in V\ : \ Tv=vs     \}.
$$
This definition gives only rise to the eigenvalues of $T$.
The other possibility is to define the left spectrum $\sigma_L(T)$, where we replace $Tv=vs$ by $Tv=sv$.
In the case of the right eigenvalues we have a nonlinear operator associated with the spectral problem. In  the case of the left spectrum, we have a linear equation,
 but in both cases there does not exist a suitable notion of
resolvent operator which preserves some sort of hyperholomorphicity.
\\
Only in the case of the $S$-spectrum, we can associate to it the
 two $S$-resolvent operators. For $s\in \rho_S(T)$, the left $S$-resolvent operator
$$
S_L^{-1}(s,T):=-(T^2-2\Re(s) T+|s|^2\mathcal{I})^{-1}(T-\overline{s}\mathcal{I}),
$$
and the right $S$-resolvent operator
$$
S_R^{-1}(s,T):=-(T-\overline{s}\mathcal{I})(T^2-2\Re(s) T+|s|^2\mathcal{I})^{-1}.
$$
are  operator-valued slice hyperholomorphic functions.
Moreover, with the notion of $S$-spectrum, the usual decompositions of the spectrum turn out to be natural, for example
in point spectrum, continuous spectrum and residual spectrum.
\\
So one of the main aims of this paper is to define the Schatten classes of
quaternionic operators using the notion of $S$-spectrum and to study their main properties.
\\
Here we point out the following fact that shows one of the main differences with respect to the complex case.
The trace of a compact complex linear operator is defined as
\begin{equation}
\tr A = \sum_{n\in\N} \langle e_n,A e_n\rangle ,
\end{equation}
where $(e_n)_{n\in\N}$ is an orthonormal basis of the Hilbert space $\hil_\C$.
It can be shown that the trace is independent of the choice of the orthonormal basis and therefore well-defined.
If we try to generalize the above definition to the quaternionic
setting, we face the problem that this invariance with respect to the choice of the orthonormal basis does not persist.
We have to define the trace in a different way.
In fact we have to restrict to operators that satisfy an additional condition:
{\em
Let $J \in \mathcal{B}(\hil)$ be an anti-selfadjoint and unitary operator.
For $p\in(0,+\infty]$, we define the $(J,p)$-Schatten class of operators $S_p(J)$ as
$$
S_p(J) := \{ T \in \B_0(\hil) :  [T,J]=0 \ \text{and}\ (\lambda_n(T))_{n\in\N}\in\ell^p\},
$$
where $(\lambda_n(T))_{n\in\N}$ denotes the sequence of singular values of $T$.
}
\\
\\
The Berezin transform is a useful tool to study operators acting on spaces of holomorphic functions
such as Bergman spaces.
In this paper we treat the case of quaternionic operators acting on  spaces of slice
hyperholomorphic functions.
\\
\\
We recall some facts of the classical theory, see \cite{Stroethoff} for more details.
To state the definition of the Berezin transform of a bounded operator we recall that:
\\
{\em
A reproducing  kernel Hilbert space on an open set $\Omega$ of  $\mathbb{C}$, is a Hilbert space
 $\mathcal{H}_\C$ of functions on $\Omega$ such that for every $w\in \Omega$ the linear functional
 $f\to f(w)$ is bounded on $\mathcal{H}_\C$.}
 \\
 When $\mathcal{H}_\C$ is a reproducing  kernel Hilbert space,
 by the Riesz representation theorem, there exists a unique $K_w\in \mathcal{H}_\C$ for every $w\in \Omega$
  such that $f(w)=\langle f,K_w\rangle$ for all $f\in \mathcal{H}_\C$. The function $K_w$ is called a
  reproducing kernel.
It is a well known fact that $K_w(z)=\sum_{\ell\in L}\overline{e_\ell}(w)e_\ell(z)$ where $\{e_\ell\ :\ \ell\in L\}$
is an orthonormal basis for $\mathcal{H}_\C$. Since $K_w(z)=\overline{K_z(w)}$, we write $
K(z,w):=K_w(z)$. The norm of $K_w$ is given by
$$
\|K_w\|^2=\langle K_w,K_w\rangle=K_w(w)=K(w,w)
$$
and when this norm is different from zero the function
$$
k_w=\frac{K_w}{\|K_w\|}
$$
is called normalized reproducing kernel at $w$. Examples of reproducing kernel Hilbert spaces
are Hardy spaces, Bergman spaces, Fock spaces and Dirichlet spaces to name a few.
\\
Keeping in mind the above definition, we define the Berezin transform of an operator as follows:
{\em Let $\mathcal{H}_\C$ be a reproducing  kernel Hilbert space of analytic functions on an open set
$\Omega$ of $\mathbb{C}$ and let $A$ be a bounded operator on $\mathcal{H}_\C$. The Berezin transform of $A$ is defined as
$$
\widetilde{A}(w) :=\langle K_w, AK_w\rangle,\ \ \ {\rm for \ all}\  w\in \Omega.
$$}
The importance of the Berezin transform is due to the possibility to deduce properties of the bounded operator $A$ from the properties of the analytic function
$\widetilde{A}(w)$. For example,
an important consequence  is that
$$
A=0\ \ {\rm if\ and \ only \ if}\ \widetilde{A}(w)=0\ \ {\rm for \ all}\  w\in \Omega.
$$
When we consider quaternionic operators, we have to replace the notion of holomorphic
 function with the notion of slice hyperholomorphic function.
Using the theory of function spaces of slice hyperholomorphic functions,
it is possible to extend classical results
to the quaternionic setting.
\\
We treat the case of weighted Bergman spaces on the unit ball $\D$ in the quaternions, whose reproducing kernel is defined as follows:
consider  $\alpha > -1$ and $q$, $w\in \D$.  For $q = q_0 + i_q q_1$ set $q_{i_w} = q_0 + i_wq_1$, where $w = w_0+ i_w w_1$, and define the slice hyperholomorphic $\alpha$-Bergman kernel as
\[
K_{\alpha}(q,w):=\frac12(1-i_qi)\frac{1}{(1-q_{i_w}\overline{w})^{2+\alpha}} + \frac12(1+i_qi)\frac{1}{(1-\overline{q_{i_w}}\,\overline{w})^{2+\alpha}}.
\]
This kernel is left slice hyperholomorphic in $q$ and anti right slice hyperholomorphic in $w$. Moreover, whenever $q$ and $w$ belong to the same complex plane, it reduces to the complex Bergman kernel.
\\
The plan of the paper is as follows:.
\\
Section 1 contains the introduction.
\\
Section 2 contains a subsection in which  we discussion two of the possible definitions
of slice hyperholomorphic functions and their properties, and a subsection that contains the main results on quaternionic operator theory.
\\
Section 3 contains the Schatten class of quaternionic operators. Precisely, we consider the singular values of compact operators, the Schatten class, the dual of the Schatten class and different characterizations of Schatten class operators.
 \\
 Section 4 contains the general theory of the Berezin transform and the case of weighted Bergman spaces with several properties.

\section{Preliminary results}

Recent works on Schur analysis in the slice hyperholomorphic setting
introduced and studied the quaternionic Hardy spaces $H_2(\Omega)$, where $\Omega$ is the quaternionic
unit ball $\mathbb D$ or the half space
$\mathbb H^+$ of quaternions with positive real part, \cite{acs1, acs2, milano}. The Hardy spaces $H^p(\mathbb B)$ for $p>2$  are considered in \cite{sarfatti}.
 The Bergman spaces are treated in \cite{cglss, CGS, CGS3} and  for the Fock spaces see \cite{Fock}.
Weighted Bergman spaces, Bloch, Besov and Dirichlet spaces on the unit ball $\mathbb B$ are studied in
\cite{BBD}.
In the following we want  discuss two possible definitions of slice hyperholomorphic functions
since in both cases the above mentioned function spaces can be defined with minor changes in the proofs.
The discussion in the next subsection will clarify the variations of slice hyperholomorphicity.

\subsection{Slice hyperholomorphic functions}
The skew-field of quaternions consists of the real vector space
$$
\H:=\left\{\xi_0 + \sum_{\ell=1}^3\xi_\ell e_\ell: \xi_\ell\in\R\right\},
$$
 where the units $e_1$, $e_2$ and $e_3$ satisfy
$e_i^2 = -1$ and $ e_\ell e_\kappa  =  - e_\ell e_\kappa$ for $\ell,\kappa \in\{1,2,3\}$ with $\ell\neq \kappa$.
The real part of a quaternion $x = \xi_0 + \sum_{\ell=1}^3\xi_\ell e_\ell$ is defined as $\Re(x) := \xi_0$,
its imaginary part as
 $\underline{x} := \sum_{\ell=1}^3\xi_\ell e_\ell$ and its conjugate as $\overline{x} := \Re(x) - \underline{x}$.
Each element of the set
\[\S := \{ x\in\H: \Re(x) = 0, |x| = 1 \}\]
 is a square-root of $-1$ and is therefore called an imaginary unit. For any $i\in\S$, the subspace
 $\C_i := \{x_0 + i x_1: x_1,x_2\in\R\}$ is isomorphic to the field of complex numbers. For $i,j\in\S$ with $i\perp j$, set $k=ij = -ji$. Then $1$, $i$, $j$ and $k$ form an
 orthonormal basis of $\H$ as a real vector space and $1$ and $j$ form an orthonormal basis of $\H$ as a left or right vector space over the complex plane $\C_i$, that is
 \[ \H = \C_i + \C_i j\quad\text{and}\quad \H = \C_i + j\C_i.\]
  Any quaternion $x$ belongs to such a complex plane: if we set
 \[i_x := \begin{cases}\underline{x}/|\underline{x}|,& \text{if  }\underline{x} \neq 0 \\ \text{any }i\in\S, \quad&\text{if }\underline{x}  = 0,\end{cases}\]
 then $x = x_0 + i_x x_1$ with $x_0 =\Re(x)$ and $x_1 = |\underline{x}|$. The set
 \[
 [x] := \{x_0 + ix_1: i\in\S\},
 \]
is a 2-sphere, that reduces to a single point if $x$ is real. We recall
 a well known result that we will need in the sequel.
\begin{lemma}\label{Turn}
Let $x,y\in\H$. Then $y\in[x]$ if and only if there exists $q\in\H$ with $|q| = 1$ such that $y = q^{-1}xq$.
\end{lemma}

The notion of slice hyperholomorphicity is the generalization of holomorphicity
to quaternion-valued functions that underlies the theory of quaternionic linear operators.
We recall the main results on slice hyperholomorphic functions.
The proofs of the results stated in this subsection can be found in the book \cite{SCalcBook}.

\begin{definition}
A set $U\subset\H$ is called
\begin{enumerate}[(i)]
\item axially symmetric if $[x]\subset U$ for any $x\in U$ and
\item a slice domain if $U$ is open, $U\cap\R\neq 0$ and $U\cap\C_i$ is a domain for any $i\in\S$.
\end{enumerate}
\end{definition}

\begin{definition}
Let $U\subset\H$ be an axially symmetric open set. A real differentiable function $f: U\to\H$ is called left slice hyperholomorphic if it has the form
\begin{equation}\label{LSHol}
 f(x) = \alpha(x_0,x_1) + i_x\beta(x_0,x_1),\quad \forall x = x_0 + i_x x_1 \in U\end{equation}
such that the functions $\alpha$ and $\beta$, which take values in $\H$, satisfy the compatibility condition
\begin{equation}\label{SymCond}
\begin{split}\alpha(x_0,x_1) &= \alpha(x_0,x_1)\\
 \beta(x_0,x_1) &= - \beta(x_0,-x_1)
 \end{split}
\end{equation}
and the Cauchy-Riemann-system
\begin{equation}\label{CR}\begin{split}
\frac{\partial}{\partial x_0} \alpha(x_0,x_1) &= \frac{\partial}{\partial x_1} \beta(x_0,x_1)\\
\frac{\partial}{\partial x_0}\beta(x_0,x_1) &= -\frac{\partial}{\partial x_1}\alpha(x_0,x_1).
\end{split}\end{equation}
A function $f:U\to\H$ is called right slice hyperholomorphic if it has the form
\begin{equation}\label{RSHol}
 f(x) = \alpha(x_0,x_1) + \beta(x_0,x_1)i_x,\quad \forall x = x_0 + i_x x_1 \in U,
 \end{equation}
such that the functions $\alpha$ and $\beta$ satisfy \eqref{SymCond} and \eqref{CR}.

The sets of left and right slice hyperholomorphic functions on $U$ are denoted by $\lhol(U)$ and $\rhol(U)$, respectively. Finally, we say that a function $f$ is left or right slice hyperholomorphic on a closed axially symmetric set $K$, if there exists an open axially symmetric set $U$ with $K\subset U$ such that $f\in \lhol(U)$ resp. $\rhol(U)$.
\end{definition}

\begin{corollary}
Let $U\subset\H$ be axially symmetric.
\begin{enumerate}[(i)]
\item If $f,g\in\lhol(U)$ and $a\in\H$, then $fa+g\in\lhol(U)$.
\item If $f,g\in\rhol(U)$ and $a\in\H$, then $af + g\in\rhol(U)$.
\end{enumerate}
\end{corollary}
On axially symmetric slice domains, slice hyperholomorphic functions can be characterized as
those functions that lie in the kernel of a slicewise Cauchy-Riemann-operator.
As a consequence, the restriction of a slice hyperholomorphic function to a complex plane can be split into two holomorphic components.
\begin{definition}
Let $f$ be a function defined on a set $U\subset\H$  and let $i\in\S$. We denote the restriction of $f$ to the complex plane $\C_i$ by $f_i$, i.e. $f_i := f|_{U\cap \C_i}$.
\end{definition}
In order to avoid confusion we point out that, in this paper, indices $i$ and $j$ always refer to restrictions of a function to the respective complex planes.
\begin{definition}
Let $U\subset\H$ be open. We define the following differential operators: for any real differentiable function $f:U\to\H$ and any ${x = x_0 + i_x x_1\in U}$, we set
\begin{align*}
\partial_i f(x) &= \frac12\left(\frac{\partial}{\partial x_0} f_{i_x}(x) - i_x \frac{\partial}{\partial x_1} f_{i_x}(x) \right)\\
\bpartial_i f(x) &= \frac12\left(\frac{\partial}{\partial x_0} f_{i_x}(x) + i_x \frac{\partial}{\partial x_1} f_{i_x}(x)\right)
\end{align*}
and
\begin{align*}
f(x)\rpartial{i} &= \frac12\left(\frac{\partial}{\partial x_0} f_i(x) - \frac{\partial}{\partial x_1} f_i(x)i_x\right)\\
 f(x)\rbpartial{i} &= \frac12\left(\frac{\partial}{\partial x_0} f_i(x) +  \frac{\partial}{\partial x_1} f_i(x)i_x \right),
\end{align*}
where the arrow $\leftarrow$ indicates that the operators $\rbpartial{i}$ and $\rpartial{i}$ act from the right.
\end{definition}
If a function depends on several variables and we want to stress that these operators act in a variable $x$, then we write $\partial_{ix}$ instead of $\partial_{i}$ etc.
\pagebreak[3]
\begin{corollary}\label{IDeriv}
Let $U\subset\H$ be open and axially symmetric.
\begin{enumerate}[(i)]
\item If $f\in\lhol(U)$, then $\bpartial_i f = 0$. If $U$ is a slice domain,
then $f\in\lhol(U)$ if and only if $\bpartial_i f = 0$.
\item  If $f\in\rhol(U)$, then $f\rbpartial{i} = 0$.  If $U$ is a slice domain,
then $f\in\rhol(U)$ if and only if $f \rbpartial{i}  = 0$.
\end{enumerate}
\end{corollary}
\begin{lemma}[Splitting Lemma]\label{SplitLem}
Let $U\subset\H$ be axially symmetric and let $i,j\in\S$ with $i\perp j$.
\begin{enumerate}[(i)]
\item If $f\in\lhol(U)$, then there exist holomorphic functions
$f_1,f_2: U\cap\C_i\to\C_i$ such that $ f_i = f_1 + f_2j$.
\item If $f\in\rhol(U)$, then there exist holomorphic functions
$f_1,f_2: U\cap\C_i\to\C_i$ such that $ f_i = f_1 + jf_2$.
\end{enumerate}
\end{lemma}
\begin{remark}
Originally, in particular in \cite{SCalcBook}, slice hyperholomorphic
functions were defined as functions that satisfy $\bpartial_i f = 0$ resp. $f\rbpartial{i}  = 0$.
In principle, this leads to a larger class of functions, but on axially symmetric slice domains
both definitions are equivalent. Indeed, on an axially symmetric slice domain, $\bpartial_if = 0$
resp. $f\rbpartial{i} = 0$ implies that $f$ satisfies the representation formula
(cf. \Cref{RepFo}), which allows for a representation of $f$ of the form \eqref{LSHol} resp. \eqref{RSHol}.

The theory of slice hyperholomorphicity was therefore only developed for functions
that are defined on axially symmetric slice domains.  However, most results do not
 directly depend on the fact that the functions are defined on a slice domain, but
 rather on their representation of the form \eqref{LSHol} resp. \eqref{RSHol}.
  Hence, the definition given in this paper seems to be more appropriate since
   it allows to extend the theory also to functions defined on sets that are
   not connected or do not intersect the real line.
\end{remark}

\begin{definition}
Let $U\subset\H$ be axially symmetric. A left slice hyperholomorphic $f(x) = \alpha(x_0,x_1) + i_x\beta(x_0,x_1)$ is called intrinsic if $\alpha$ and $\beta$ are real-valued. We denote the set of intrinsic functions on $U$ by $\intrin(U)$.
\end{definition}
Note that an intrinsic function is both left and right slice hyperholomorphic because $\beta(x_0,x_1)$ commutes with the imaginary unit $i_x$. The converse is not true: the constant function $x\mapsto b\in\H\setminus\R$ is left and right slice hyperholomorphic, but it is not intrinsic.

The importance of the class of intrinsic functions is due to the fact that the multiplication and composition with intrinsic functions preserve slice hyperholomorphicity. This is not true for arbitrary slice hyperholomorphic functions.
\begin{corollary} Let $U\subset\H$ be axially symmetric.
\begin{enumerate}[(i)]
\item If $f\in\intrin(U)$ and $g\in\lhol(U)$, then $fg\in\lhol(U)$. If $f\in\rhol(U)$ and $g\in\intrin(U)$, then $fg\in\rhol(U)$.
\item If $g\in\intrin(U)$ and $f\in\lhol(g(U))$, then $f\circ g\in \lhol(U)$. If $g\in\intrin(U)$ and $f\in\rhol(g(U))$, then $f\circ g\in \rhol(U)$.
\end{enumerate}
\end{corollary}

Important examples of slice hyperholomorphic functions are power series with quaternionic coefficients: series of the form $\sum_{n=0}^{+\infty}x^na_n$ are left slice hyperholomorphic and series of the form $\sum_{n=0}^{\infty} a_nx^n$ are right slice hyperholomorphic on their domain of convergence. A power series is intrinsic if and only if its coefficients are real.

Conversely, any slice hyperholomorphic function can be expanded into a power series at any real point.
\begin{definition}
The slice-derivative of a function $f\in\lhol(U)$ is defined as
\[\sderiv  f(x) = \lim_{\C_{i_x}\ni s\to x} (s-x)^{-1}(f(s)-f(x)), \]
where $\lim_{\C_{i_x}\ni s\to x} g(s)$ is the limit as $s$ tends to $x = x_0 + i_xx_1 \in U$ in $\C_{i_x}$.
The slice-derivative of a function $f\in\rhol(U)$ is defined as
\[  f\rsderiv(x) = \lim_{\C_{i_x}\ni s\to x} (f(s)-f(x))(s-x)^{-1}. \]
\end{definition}
If a function depends on several variables and we want to stress that we take the slice-derivative in the variable $x$, then we write $\sderiv[x]$ and $\rsderiv[x]$ instead of $\sderiv$ and~$\rsderiv$.
\begin{corollary}
The slice derivative of a left (or right) slice hyperholomorphic function is again left (or right) slice hyperholomorphic. Moreover, it coincides with the derivative with respect to the real part, that is
\[\sderiv  f(x) = \frac{\partial}{\partial x_0} f(x) \quad \text{resp.}\quad f\rsderiv(x) = \frac{\partial}{\partial x_0} f(x). \]
\end{corollary}
\begin{theorem}
If $f$ is left slice hyperholomorphic on the ball $B_r(\alpha)$ with radius $r$ centered at $\alpha\in\R$, then
\[f(x) = \sum_{n=0}^{+\infty} (x-\alpha)^n \frac{1}{n!}\sderiv ^n f(\alpha)\quad\text{for }\ \ x\in B_r(\alpha).\]
If $f$ is right slice hyperholomorphic on $B(r,\alpha)$, then
\[f(x) = \sum_{n=0}^{+\infty}\frac{1}{n!} f\rsderiv ^n(\alpha) (x-\alpha)^n \quad\text{for }\ \ x\in B_r(\alpha).\]
\end{theorem}

As a consequence of the slice structure of slice hyperholomorphic functions, their values are uniquely determined by their values on an arbitrary complex plane. Consequently, any function that is holomorphic on a suitable subset of a complex plane possesses an unique slice hyperholomorphic extension.
\begin{theorem}[Representation Formula]\label{RepFo}
Let $U\subset\H$ be axially symmetric and let $i\in\S$. For any $x = x_0 + i_x x_1\in U$ set $x_i := x_0 + ix_1$. If $f\in\lhol(U)$. Then
\[f(x) = \frac12(1-i_xi)f(x_i) + \frac12(1+i_xi)f(\overline{x_i}).\]
If $f\in\rhol(U)$, then
\[f(x) = f(x_i)(1-ii_x)\frac12 + f(\overline{x_i})(1+ii_x)\frac12.\]
\end{theorem}

\begin{corollary}\label{extLem}
Let $i\in\S$ and let $f:O\to\H$ be real differentiable, where $O$ is a domain in $\C_i$ that is symmetric with respect to the real axis.
\begin{enumerate}[(i)]
\item The axially symmetric hull $[O]: = \bigcup_{z\in O}[z]$ of $O$ is an axially symmetric slice domain.
\item If $f$ satisfies $\frac{\partial}{\partial x_0}f + i \frac{\partial}{\partial x_1} f= 0$, then there exists a unique left slice hyperholomorphic extension $\ext_L(f)$ of $f$ to $[O]$.
\item If $f$ satisfies $\frac{\partial}{\partial x_0}f +  \frac{\partial}{\partial x_1}fi = 0$, then there exists a unique right slice hyperholomorphic extension $\ext_R(f)$ of $f$ to $[O]$.
\end{enumerate}
\end{corollary}
\begin{remark}
If $f$ has a left and a right slice hyperholomorphic extension, they do not necessarily coincide. Consider for instance the function $z\mapsto bz$ on $\C_i$ with a constant $b\in\C_i\setminus\R$. Its left slice hyperholomorphic extension to $\H$ is $x\mapsto xb$, but its right
slice hyperholomorphic extension is $x\mapsto bx$.
\end{remark}

\subsection{Quaternionic linear operators}
In this section, we consider bounded linear operators on a separable quaternionic Hilbert space, even thought
some definitions and results hold also for quaternionic Banach spaces.

\begin{definition}
Let $\hil$ be a quaternionic right vector space together with a scalar product  $\langle\cdot,\cdot\rangle: \hil\times\hil\to \H$ with the following properties:
\begin{enumerate}[(i)]
\item positivity: $\langle x, x\rangle \geq 0 $ and $\langle x,x\rangle  = 0$ if and only if $x = 0$
\item right-linearity: $\langle x, ya + z \rangle = \langle x,y\rangle a + \langle x,z\rangle $ for all $x,y,z\in\hil$ and all $a\in\H$.
\item quaternionic hermiticity: $\langle x,y\rangle = \overline{\langle y,x\rangle}$ for all $x,y \in\hil$.
\end{enumerate}
If $\hil$ is complete with respect to the norm
\[\|x\|_{\hil} := \sqrt{\langle x, x\rangle},\]
then $(\hil, \langle\cdot,\cdot\rangle)$ is called a quaternionic Hilbert space.
\end{definition}
Terms such as orthogonality, orthonormal basis etc. are defined as in the complex case. In the following, we shall always assume that $\hil$ is separable.

We consider now bounded quaternionic right linear operators on $\hil$.
\begin{definition} A quaternionic right linear operator $T:\hil\to\hil$ is called bounded if $\|T\| := \sup_{\|x\| = 1} \|Tx\| <+\infty$.
We denote the set of all bounded quaternionic right linear operators on $\hil$ by $\B(\hil)$.
\end{definition}

The space $\B(\hil)$  together with the operator norm is a real Banach space.
Observe that there does not exists any natural multiplication of operators with scalars in
$\H\setminus\R$ because there is no left-multiplication defined on $\hil$. Hence, the operator $Ts$,
 which is supposed to act as $Ts(v) = T(sv)$, has no meaning.

Let $T\in\boundOP(\hil)$. For $s\in\H$, we set
\[ Q_s(T) := T^2 - 2\Re(s)T + \id |s|^2,\]
where $\id$ denotes the identity operator.
\begin{definition}\label{SSpec}
We define the $S$-resolvent set  of an operator $T\in\boundOP(\hil)$ as
\[\rho_S(T):= \{ s\in\H: Q_s(T)^{-1}\in\boundOP(\hil)\}\]
and the $S$-spectrum of $T$ as
\[\sigma_S(T):=\H\setminus\rho_S(T).\]
\end{definition}
\begin{remark}
For operators acting on a two-sided quaternionic Banach space, the left and right $S$-resolvent operators are defined by formally replacing the variable $x$ in the left and right Cauchy kernel by the operator $T$. This motivates the definition of the $S$-resolvent set and the $S$-spectrum as it is done in \Cref{SSpec}, cf. \cite{SCalcBook}.
\end{remark}
\begin{theorem}[See \cite{SCalcBook}]
Let $T\in\boundOP(\hil)$. Then $\sigma_S(T)$ is an axially symmetric,
compact and nonempty subset of $\overline{B_{\|T\|}(0)}$.
\end{theorem}
In \cite{GMP}, the following natural partition of the $S$-spectrum was introduced:
\begin{enumerate}
\item the spherical point spectrum of $T$:
\[\sigma_{pS}(T) := \{ s\in\H: \ker(Q_s(T))\neq\{0\} \}.\]
\item the spherical residual spectrum of $T$:
\[\sigma_{rS}(T) := \{ s\in\H: \ker(Q_s(T)) = \{0\},\, \overline{\ran(Q_s(T))} \neq \hil)\]
\item the spherical continuous spectrum of $T$:
\[\sigma_{cS}(T) := \{ s\in\H: \ker (Q_S(T)) = \{0\},\ \overline{\ran(Q_s(T))}
= \hil,\ Q_s(T)^{-1}\notin \boundOP(\hil)\}.\]
\end{enumerate}

\begin{theorem}
Let $T\in\boundOP(\hil)$. Then $s\in\sigma_{pS}(T)$ if and only if $s$ is a right
 eigenvalue of $T$, i.e. there exits $x\in \hil$ such that $Tx = xs$.
\end{theorem}

The adjoint of an operator and related properties are defined analogue to the complex case.
\begin{definition}
Let $T\in\boundOP(\hil)$. The adjoint $T^*$ of $T$ is the unique operator that satisfies
\[ \langle T^* x, y\rangle = \langle x, Ty\rangle\quad\forall x,y\in\hil.\]
\end{definition}

\begin{definition}
An operator $T\in\boundOP(\hil)$ is called
\begin{enumerate}[(i)]
\item selfadjoint, if $T^* = T$
\item anti-selfadjoint, if $T^* = - T$
\item positive if $\langle x, Tx\rangle \geq 0$ for all $x\in\hil$
\item normal if $T^*T = TT^*$
\item unitary if $T^* = T^{-1}$.
\end{enumerate}
\end{definition}
If $T\in\boundOP(\hil)$ is positive, then there exists a unique positive operator $S$ such that $S^2 = T$. We denote $\sqrt{T} := S$. Moreover, for any operator $T\in\boundOP(\hil)$, the operator $T^*T$ is positive and we can define
\[|T| := \sqrt{T^*T},\]
which allows to prove the polar decomposition theorem.
\begin{theorem}\label{polar}
Let $\hil$ be a quaternionic Hilbert space and let $T \in \mathcal{B}(\hil)$. Then there exist two unique operators $W$ and $P$ in $\mathcal{B}(\hil)$ such that
\begin{enumerate}[(i)]
\item $T = WP$
\item $P \geq 0$
\item $\ker(P) \subset \ker(W)$
\item $\forall u \in \ker(P)^\perp: \|Wu\| = \|u\|$.
\end{enumerate}
The operators $W$ and $P$ have the following properties
\begin{enumerate}[(a)]
\item $P = |T|$
\item If $T$ is normal then $W$ defines a unitary operator in $\mathcal{B}(\overline{\ran(T)})$
\item $W$ is (anti) self-adjoint if  $T$ is.
\end{enumerate}
\end{theorem}
The above theorem is mentioned several times in the literature, for a proof see \cite{GMP}.
In order to investigate linear operators on a quaternionic Hilbert space, it is often useful to define a complex structure on the space that allows to write every vector in terms of two components that belong to a certain complex Hilbert space. If this complex structure is chosen appropriately, then the considered operator is the natural extension of a complex linear operator on the component space.
\begin{definition}[See G. Emch \cite{12}]
Let $J\in\boundOP(\hil)$ be anti-selfadjoint and unitary and let $i\in\S$. We define the complex subspaces
\begin{equation*}
\hil_+^{Ji} := \{ x \in \hil: Jx = xi\}\quad\text{and}\quad \hil_-^{Ji} := \{ x \in \hil: Jx = -xi\}.
\end{equation*}
\end{definition}
\begin{theorem}
$\hil_+^{Ji}$ is a nontrivial complex Hilbert space over $\C_i$ with respect to the structure induced by $\hil$: its sum is the sum of $\hil$, its complex scalar multiplication is the right scalar multiplication of $\hil$ restricted to $\C_i$ and its complex scalar product is the restriction of the scalar product of $\hil$ to $\hil_+^{Ji}\times \hil_+^{Ji}$.
An analogous statement holds true for $\hil_-^{Ji}$.
\end{theorem}

\begin{theorem}\label{ComplexONB}
Every orthonormal basis $(e_n)_{n\in\N}$ of the complex Hilbert space $\hil_+^{Ji}$ is also an orthonormal basis of $\hil$. Moreover, it is
\begin{equation*}
Jx = \sum_{n\in\N} e_n i \langle e_n, x\rangle,\quad \forall x\in\hil.
\end{equation*}
\end{theorem}
For a proof of the above two results see \cite{GMP}.
Observe that \Cref{ComplexONB} implies that we can write $x\in\hil$ as $x = x_1+ x_2j$ with $x_1,x_2\in\hil_+^{Ji}$. Moreover, it justifies considering $J$ as a left multiplication with the imaginary unit $i\in\S$. We may set $ix := Jx$ for any $x\in\hil$ and extend this left scalar multiplication to the entire complex field $\C_i$ by setting $(a_0 + i a_1)x = xa_0 + Jxa_1$. Since
\[ix = Jx = Jx_0 + Jx_1j = x_0i + x_1ij,\]
we then obtain \( ax = x_0a + x_1 aj\) for $a\in\C_i$.

Note that the choice of the imaginary unit $i$ is arbitrary, but that different imaginary units will lead to different left scalar multiplications, which are of course only defined for scalars in the respective complex plane!

We discuss now the relation between quaternionic linear operators on $\hil$ and complex linear operators on $\hil_+^{Ji}$. In order to avoid confusion, we distinguish complex linear operators by a $\C$-subscript.
\begin{theorem}\label{OPExtension}
Let $J$ be an anti-selfadjoint, unitary operator on $\hil$ and take $i \in \mathbb{S}$. For every bounded $\C_i$-linear operator $T_{\C}\in\boundOP( \hil^{Ji}_+)$, there exists a unique right $\field{H}$-linear operator $\widetilde{T_{\C}}\in\boundOP(\hil)$ that satisfies
\[ \widetilde{T_{\C}}x = T_{\C}x,\quad \forall x\in\hil_{+}^{Ji}.\]
The following facts also hold:
\begin{enumerate}[(a)]
\item We have $\|\widetilde{T_{\C}}\| = \|T_{\C}\|.$
\item It is $[\widetilde{T_{\C}},J]=0$, where $[\widetilde{T_{\C}},J] := \widetilde{T_{\C}}J - J\widetilde{T_{\C}}$ denotes the commutator of $\widetilde{T_{\C}}$ and $J$.
\item An operator $V\in\boundOP(\hil)$ is equal to $\widetilde{T_{\C}}$ for some $\C_i$-linear operator $T_{\C}\in\boundOP(\hil_+^{Ji})$ if and only if $[J,V] = 0$.
\item We have  $(\widetilde{T_{\C}})^* = \widetilde{T_{\C}^*}$.
\end{enumerate}
Furthermore, given another $\C_i$-linear operator $S_{\C}\in\boundOP( \hil_+^{Ji})$, we have
\begin{enumerate}[(a)]\addtocounter{enumi}{4}
\item $\widetilde{S_{\C}T_{\C}} = \widetilde{S_{\C}} \widetilde{T_{\C}}$.
\item If $S_{\C}$ is a right (resp. left) inverse of $T_{\C}$, then $\widetilde{S_{\C}}$ is a right (resp. left) inverse of $\widetilde{T_\C}$.
\end{enumerate}
\end{theorem}
The above theorem generalizes some results of G. Emch in Section 3 of \cite{12} and formulates them in a modern language.
\begin{definition}\label{B=JJJ}
For an anti-selfadjoint unitary operator $J$, we define the set
\[\B_J(\hil) := \{T\in\B(\hil): [T,J] = 0\}.\]
\end{definition}
\begin{remark}\label{BJIso}
For $T\in\boundOP_J(\hil)$ and $a = a_0 + i a_1\in\C_i$ set $aT = a_0 T + a_1 JT$. If we consider $J$ as a left-multiplication with some imaginary unit $i \in \mathbb{S}$, then  $\B_J(\hil)$ turns into a Banach algebra over $\C_i$ that is isometrically isomorphic to $\B(\hil^{Ji}_+)$. Isometric isomorphims between these two spaces are given by
\begin{align}\label{ResLift}
\res{Ji}:\begin{cases} \B_J(\hil)&\rightarrow \B(\hil_+^{Ji})\\ T &\mapsto T \big|_{\hil_+^{Ji}}\end{cases}
\qquad
\lift{Ji}: \begin{cases} \B(\hil_+^{Ji}) &\rightarrow \B_J(\hil)\\ T_{\C} &\mapsto \widetilde{T_{\C}}\end{cases},
\end{align}
where $\widetilde{T_{\C}}$ is the unique operator obtained from \Cref{OPExtension}.
\end{remark}

\begin{corollary}\label{CorResLift}
An operator $T\in\boundOP_J(\hil)$ is selfadjoint, anti-selfadjoint, positive, normal or unitary if and only if $\res{Ji}(T)$ is.
\end{corollary}

\section{Schatten classes of quaternionic linear operators}
In order to introduce Schatten classes of quaternionic linear operators, we need to define the singular values of compact quaternionic linear operators. For analogue results in the complex case see \cite{kehe}.

\subsection{Singular values of compact operators}
\begin{definition}
A right linear operator is called compact if it maps bounded sequences to sequences that admit convergent subsequences. We denote the set of all compact quaternionic right linear operators on $\hil$ by $\B_0(\hil)$.
\end{definition}

We recall the spectral theorem for compact quaternionic operators (see \cite{spectcomp}); note that this is a special case of the spectral theorem for arbitrary normal operators, which was recently established in \cite{ack}.
\begin{theorem}\label{CompSpec}
Given a normal operator $T \in \B_0(\hil)$ with spherical point spectrum $\sigma_{pS}(T)$ there exists a Hilbert basis $\mathcal{N}\subset \hil$ made of eigenvectors of $T$ such that
\begin{align}\label{SpecDec}
Tx &= \sum _{z \in \mathcal{N}}  z \lambda_z \inpr{z}{x}~~~~~~\forall x \in \hil,
\end{align}
where $\lambda_z \in \H$ is an eigenvalue relative to the eigenvector $z$ and if $\lambda_z \neq 0$ then there are only a finite number of distinct $z' \in \mathcal{N}$ such that $\lambda_z = \lambda_{z'}$. Moreover the values $\lambda_z$ are at most countably many and $0$ is their only possible accumulation point.
\end{theorem}
Recall that we are considering only separable Hilbert spaces. Hence, $\mathcal{N}$ is always countable and we can write $\mathcal{N} = \{e_n: n\in\N\}$.
\begin{remark} \label{CompSpecPlane}
Observe that the eigenvalues $(\lambda_n)_{n\in\N}$ in the spectral decomposition of a compact quaternionic linear operator are in general not unique. Indeed, if $Te_n = e_n\lambda_n$ and $q\in\H$ with $|q| = 1$, then
\[T(e_nq) = (Te_n)q = e_n\lambda_n q = (e_nq)q^{-1}\lambda_n q\]
and hence the eigenpair $(e_n,\lambda_n)$ can be replaced by $(e_nq, q^{-1}\lambda_n q)$.
 However, $q^{-1}\lambda_nq \in [\lambda_n]$ and consequently at least the eigenspheres $[\lambda_n]$
 are uniquely determined. In particular, the eigenvalues can be chosen such that
 $\lambda_n\in\C_i^+:=\{s = s_0 + is_1\in \C_i: s_1\geq0\}$ for a given imaginary unit $i\in\S$, cf. Lemma \ref{Turn}.
\end{remark}
\begin{remark}\label{}
Using the spectral theorem, one can define a functional calculus for so-called slice functions. We just need the special case of fractional powers of a positive operator compact operator $T$: consider its spectral decomposition $T = \sum_{n\in\N}e_n\lambda_n\langle e_n,\cdot\rangle$. For $p>0$,  the operator $T^p$ is defined as
\[T^p = \sum_{n\in\N} e_n\lambda_n \langle e_n, \cdot\rangle. \]
\end{remark}

Consider now an arbitrary compact operator $T$. Then the operator $|T|$ is normal and, combining \Cref{CompSpec} applied to $|T|$ with the polar
decomposition, \Cref{polar}, we are capable of finding a Hilbert-basis $(e_n)_{n\in\N}$ and an orthonormal set $(\sigma_n)_{n\in\N}$ in $\hil$ such that
\begin{align}\label{SVD}
Tx = \sum _{n\in\N} \sigma_n\lambda_n \inpr{e_n}{x}\qquad \forall x \in \hil,
\end{align}
where the $\lambda_n\in\R^+$ are the eigenvalues of the operator $|T|$ in decreasing order, the vectors $(e_n)_{n\in\N}$ form an eigenbasis of $|T|$ and $\sigma_n = We_n$ with $W$ unitary such that $T = W|T|$.

\begin{definition}We call the set $\{\lambda_n  \}_{n \in \N}$
the set of singular values of $T$ and the representation \eqref{SVD} the singular value decomposition of $T$.
\end{definition}

The following Rayleigh's equation gives a characterization of the singular values of $T$. Since the proof follows the lines of the complex case, we omit it and refer to \cite{ds2}.

\begin{lemma}
Let $T$ be a positive, compact operator on $\hil$ and let
\begin{align}\label{ASESPECD}
Tx &= \sum _{n \in \N}  e_n \lambda_n \inpr{e_n}{x}
\end{align}
be the spectral decomposition of $T$, where the eigenvalues $\{\lambda_n\}_{n\in\N}$ are given in descending order. Then
\[
\lambda_{n+1} = \min_{y_1,\ldots, y_n} \max_{\begin{subarray}{c}\inpr{x}{y_i}=0\\ i=1,\ldots,n\end{subarray}}\frac{\inpr{x}{Tx}}{\|x\|^2}.
\]
\end{lemma}
Also in the quaternionic setting  we have:
\begin{lemma}
Let $T$ be a positive, compact operator on $\hil$. Then
\begin{equation}
\label{MinMaxNorm}
\lambda_{n+1} = \min_{y_1,\ldots, y_n} \max_{\begin{subarray}{c}\inpr{x}{y_i}=0\\ i=1,\ldots,n\\\|x\|=1\end{subarray}} \|Tx\|.
\end{equation}
\end{lemma}
\begin{proof}
In fact
\begin{align*}
\lambda_{n+1} = \min_{y_1,\ldots, y_n} \max_{\begin{subarray}{c}\inpr{x}{y_i}=0\\ i=1,\ldots,n\\\|x\|=1\end{subarray}}\inpr{x}{Tx}.
\end{align*}
  using the Cauchy-Schwarz inequality we see that
\begin{align*}
\inpr{x}{Tx} \leq \|x\| \|Tx\|
\end{align*}
and since we assumed in the above lemma that $T$ is positive, we can consider its spectral decomposition \eqref{ASESPECD} in order to see that
\[
\begin{split}
\lambda_{n+1} &= \min_{y_1,\ldots, y_n}
\max_{\begin{subarray}{c}\inpr{x}{y_i}=0\\ i=1,\ldots,n\\\|x\|=1
\end{subarray}}
\inpr{x}{Tx}
\leq \min_{y_1,\ldots, y_n} \max_{\begin{subarray}{c}\inpr{x}{y_i}=0\\ i=1,\ldots,n\\\|x\|=1\end{subarray}} \|Tx\|
\leq \max_{\begin{subarray}{c}\inpr{x}{e_i}=0\\ i=1,\ldots,n\\\|x\|=1\end{subarray}} \|Tx\|
\\
&
= \max_{\begin{subarray}{c}\inpr{x}{e_i}=0\\ i=1,\ldots,n\\\|x\|=1
\end{subarray}} \sqrt{\sum_{m \geq n+1}\lambda_m^2 |\inpr{x}{e_m}|^2}
\leq \lambda_{n+1}\max_{\begin{subarray}{c}\inpr{x}{e_i}=0\\ i=1,\ldots,n\\\|x\|=1\end{subarray}}
\sqrt{\sum_{m \geq n+1}|\inpr{x}{e_m}|^2}
\\
&\leq \lambda_{n+1}.
\end{split}
\]
\end{proof}
The singular values of an arbitrary compact operator $T$ are the
eigenvalues of the positive operator $|T|$. Since the eigenvalues of $|T|$
can be obtained via the formula \eqref{MinMaxNorm} and since
$\|Tx\| = \|W |T| x\| = \| |T| x\|$ by \Cref{polar}, the singular values of $T$ also satisfy~\eqref{MinMaxNorm}.

The formula \eqref{MinMaxNorm} also allows us to deduce the following corollary, just as in the complex case, cf. \cite[Theorem~1.34]{kehe}.
\begin{corollary}\label{singvalnorm}
Let $T$ be a  compact operator on $\hil$. Its singular values satisfy
\begin{align}\label{OPCharac}
\lambda_{n+1} = \inf_{F \in \mathcal{F}_n}\|T-F\|,
\end{align}
where $\mathcal{F}_n$ is the set of all linear operators on $\hil$ with rank less than or equal to~$n$.
\end{corollary}

The following proposition is an immediate consequence of \eqref{MinMaxNorm}, cf. \cite[Corollary~XI.9.3]{ds2} for the complex case.
\begin{corollary}\label{singvalsom}
Let $T_1$ and $T_2$ be compact linear operators on $\hil$. Then, for every nonnegative $n,m\in\N_{0}$, we have that
\begin{align*}
\lambda_{n+m+1}(T_1+T_2) \leq \lambda_{n+1}(T_1) + \lambda_{m+1}(T_2)
\end{align*}
and
\begin{align*}
\lambda_{n+m+1}(T_1T_2) \leq \lambda_{n+1}(T_1)\lambda_{m+1}(T_2).
\end{align*}
\end{corollary}

\subsection{Definition of the Schatten class}

For $p\in(0,+\infty)$ the Schatten $p$-class $S_p(\hil_\C)$ of operators
on a complex Hilbert space $\hil_\C$ is defined as the set of compact operators
whose singular value sequences are $p$-summable.
For $p\geq 1$, the Schatten $p$-norm, which assigns to each operator the $\ell^p$-norm of its singular value sequence,
turns $S_p(\hil_\C)$ into a Banach space.
However, the analogue approach does not make sense in the quaternionic setting.
Major problems appear in particular when trying to define the trace of an operator,
which plays a crucial role in the classical theory. The trace of a compact complex linear operator is defined as
\begin{equation}\label{CTrace}\tr A = \sum_{n\in\N} \langle e_n,A e_n\rangle ,\end{equation}
where $(e_n)_{n\in\N}$ is an orthonormal basis of the Hilbert space $\hil_\C$.
It can be shown that the trace is independent of the choice of the orthonormal basis and therefore well-defined.

In general, this is not true in the quaternionic setting. Consider for example a
 compact normal operator $T$ and its spectral decompositions
\[ T =  \sum_{n\in\N} e_n\lambda_n\langle e_n, \cdot\rangle =
\sum_{n\in\N}\tilde{e}_n\tilde{\lambda}_n\langle \tilde{e}_n, \cdot\rangle,\]
where the orthonormal system $(e_n)_{n\in\N}$ is such that the corresponding eigenvalues
$(\lambda_n)_{n\in\N}$ belong to the complex halfplane $\C_i^+$ and the orthonormal
system $(\tilde{e}_n)_{n\in\N}$ is such that the corresponding eigenvalues
$(\tilde{\lambda}_n)_{n\in\N}$ belong to the complex halfplane $\C_j^+$ with $i,j\in\S$
and  $i\neq j$, cf. \Cref{CompSpecPlane}. Moreover, assume that at least one eigenvalue has
nonzero imaginary part. Then
\[ \sum_{n\in\N} \langle T e_n, e_n\rangle = \sum_{n\in\N} \lambda_n\langle e_n, e_n\rangle
= \sum_{n\in\N} \lambda_n \in\C_i^+\setminus{\R}, \]
but
\[ \sum_{n\in\N} \langle T \tilde{e}_n, \tilde{e}_n\rangle = \sum_{n\in\N}
\tilde{\lambda}_n\langle\tilde{e}_n, \tilde{e}_n\rangle = \sum_{n\in\N}
\tilde{\lambda}_n \in\C_j^+\setminus{\R}.\]
Obviously, it is
 \[\sum_{n\in\N} \langle T e_n, e_n\rangle  \neq \sum_{n\in\N} \langle T \tilde{e}_n, \tilde{e}_n\rangle.\]

 In contrast to the classical theory, we have to restrict ourselves to operators
 that satisfy an additional restriction: compatibility with a chosen complex left-multiplication.

\begin{definition}
Let $J \in \mathcal{B}(\hil)$ be an anti-selfadjoint and unitary operator.
For $p\in(0,+\infty]$, we define the $(J,p)$-Schatten class of operators $S_p(J)$ as
\begin{align*}
S_p(J) := \{ T \in \B_0(\hil) :  [T,J]=0 \ \text{and}\ (\lambda_n(T))_{n\in\N}\in\ell^p\},
\end{align*}
where $(\lambda_n(T))_{n\in\N}$ denotes the sequence of singular values of $T$
and $\ell^p$ denotes the space of $p$-summable resp. bounded sequences.
For $T\in S_p(J)$ , we define
\begin{equation}\label{pNorm}
\|T\|_p = \left(\sum _{n \in \N} |\lambda_n(T)|^p \right)^{\frac{1}{p}}\qquad\text{if }p\in(0,+\infty)
\end{equation}
and
\begin{equation*}
\|T\|_{p} = \sup_{n\in\N}\lambda_n(T) = \|T\|\qquad\text{if } p = +\infty.
\end{equation*}
\end{definition}
Obviously, $S_p(J)$ is a subset of $\B_J(\hil)$, see Definition \ref{B=JJJ}. The following lemma shows
that it is isometrically isomorphic to the Schatten $p$-class on $\hil^{Ji}_+$ if $J$
is considered as a left-multiplication with $i\in\S$.
\begin{lemma} Let $J$ be a unitary and anti-selfadjoint operator and let $i\in\S$.
For all $T \in \mathcal{B}_0(\hil)$ such that $[T,J]=0$ and for each $n \in \N$, we have that
\begin{align*}
\lambda_n(T) = \lambda_n'\left(\res{Ji}(T)\right),
\end{align*}
where $\lambda_n'(A_{\C})$ is the $n$-th singular value of a
$\C_i$-linear oper\-ator~${A_{\C} \in \mathcal{B}(\hil_+^{Ji})}$.
\end{lemma}
\begin{proof}
Since $T$ is compact, its restriction
\(
\res{Ji}(T): \hil_+^{Ji} \rightarrow \hil_+^{Ji}
\)
 is a compact operator on $\hil_+^{Ji}$. Hence, we can find orthonormal
 sets $(e_n)_{n \in \N}$ and $(\sigma_n)_{n \in \N}$ in $\hil_+^{Ji}$ such that
\begin{align*}
Tx = \sum_{n \in \N} \sigma_n \lambda_n'\left(T \big|_{\hil_+^{Ji}}\right) \inpr{e_n}{x}~~~~~~\forall x \in \hil_+^{Ji}.
\end{align*}
By the uniqueness of the extension, it follows that this expression holds for all
$x \in \hil$ and we deduce  $\lambda_n(T) = \lambda_n'(\res{Ji}(T))$.

\end{proof}
\begin{corollary}\label{SchattenIso}
Let $J$ be an anti-selfadjoint unitary operator on $\hil$ and consider
it as a left-multiplication with $i\in\S$. For any $p\in(0,+\infty]$,
the space $S_p(J)$ is isomorphic to $S_p(\hil^{Ji}_+)$, the Schatten
$p$-class on $\hil^{Ji}_+$. An isomorphism between these spaces is
given by $T\mapsto \res{Ji}(T)$. Moreover, $\|T\|_p = \|\res{Ji}(T)\|_p$.
In particular, if $p\in[1,+\infty]$, then $\|\cdot \|_p$ is actually a
norm and $S_p(J)$ is a Banach space over $\C_i$ that is even isometrically isomorphic to $S_p(\hil^{Ji}_+)$.
\end{corollary}
Finally, as an immediate consequence of \Cref{singvalsom}, we obtain
as in the complex case that any $S_p(J$) is an ideal of $\B_J(\hil)$.
\begin{corollary}\label{ideal}
Let $p\in(0,+\infty]$. Then $S_p(J)$ is a two-sided ideal of $\B_J(\hil)$,
i.e. $ST$ and $TS$ belong to $S_p(J)$ whenever $T\in S_p(J)$ and $S\in \B_J(\hil)$.
\end{corollary}

\subsection{The dual of the Schatten class}
In the following, we fix a unitary, anti-selfadjoint operator $J$ and consider it as a left-multiplication with $i\in\S$.
We define the trace of an operator $T \in S_1(J)$ and establish some elementary results in order to determine the dual space of $S_p(J)$.
\begin{definition}We define the $Ji$-trace of an operator $T \in S_1(J)$ as
\begin{align*}
\Tr{Ji}(T) := \operatorname{tr}\left(\res{Ji}(T)\right)
\end{align*}
where $\operatorname{tr\left(\res{Ji}(T)\right)}$ denotes the classical trace of a complex linear operator $\res{Ji}(T)$ as defined in \eqref{CTrace}.
\end{definition}
\begin{corollary}\label{TCy}
If $T\in S_1(J)$, then
\begin{equation}\label{TraceONB}
\Tr{Ji}(T) = \sum_{n\in\N} \langle e_n, Te_n\rangle
\end{equation}
for any orthornormal basis $(e_n)_{n\in\N}$ of $\hil_+^{Ji}$.
\end{corollary}
If $T$ is positive or selfadjoint, then also in the quaternionic setting there are no restrictions on the choice of the orthonormal basis in \eqref{TraceONB}.
\begin{lemma}\label{TracePos}
If $T$ is a positive and compact operator on $\hil$ with singular values $(\lambda_n)_{n\in\N}$
then
\begin{align*}
\sum _{n \in \N} \lambda_n = \sum _{n \in \N} \inpr{e_n}{Te_n}
\end{align*}
for each orthonormal basis $(e_n)_{n\in\N}$ of $\hil$. If in particular $T\in S_1(J)$, then \eqref{TraceONB} holds true for any orthonormal basis of $\hil$.
\end{lemma}
\begin{proof} Since $T$ is compact and positive, we have $T = \sum_{n\in\N} \eta_n\lambda_n\langle \eta_n,\cdot\rangle$ for some orthonormal basis $(\eta_n)_{n\in\N}$ of $\hil$ by \Cref{CompSpec}. For any orthonormal basis $(e_n)_{n\in\N}$ of $\hil$, we therefore have
\begin{align*}
\inpr{e_n}{Te_n} = \sum_{m\in\N} \lambda_m | \inpr{\eta_m}{e_n}|^2.
\end{align*}
Using Fubini's theorem and Parseval's identity $\|x\|^2 =   \sum _{n\in\N}| \inpr{e_n}{x}|^2$, we are left with
\begin{align*}
\sum _{n \in \N} \inpr{e_n}{Te_n} &=\sum_{n\in\N}  \sum_{m \in\N} \lambda_m   |\inpr{\eta_m}{e_n}|^2 \\
&= \sum _{m \in\N} \lambda_m  \sum _{n\in\N}| \inpr{\eta_m}{e_n}|^2= \sum _{m \in \N} \lambda_m.
\end{align*}
\end{proof}

\begin{corollary}
Let $T\in S_1(J)$ be selfadjoint. Then \eqref{TraceONB} holds true for any orthonormal basis of $\hil$.
\end{corollary}
\begin{proof}
By \Cref{CorResLift}, the operator $\res{Ji}(T)$ is a bounded selfadjoint operator on $\hil_+^{Ji}$ and can therefore be decomposed into $\res{Ji}(T) = T_{+,\C} - T_{-,\C}$ with positive operators $T_{+,\C},T_{-,\C}\in\boundOP(\hil_+^{Ji})$. If we set $T_{\pm} = \lift{Ji}(T_{\pm,\C})$, then  $T = T_{+} - T_{-}$ with positive operators $T_{+},T_{-}\in\boundOP_J(\hil)$. The additivity of the $Ji$-trace and \Cref{TracePos} imply
\begin{align*}
\Tr{Ji}(T) &= \Tr{Ji}(T_+) - \Tr{Ji}(T_-) \\
&= \sum_{n\in\N} \langle e_n, T_+e_n\rangle - \sum_{n\in\N} \langle e_n, T_-e_n\rangle = \sum_{n\in\N} \langle e_n, Te_n\rangle
\end{align*}
for any orthonormal basis $(e_n)_{n\in\N}$ of $\hil$.

\end{proof}

Observe that $\Tr{Ji}(T)$ of course depends on the imaginary unit $i\in\S$. However, as the next result shows, the choice of the imaginary unit $i$ has no essential impact.
\begin{lemma}
Let $T\in S_1(J)$ and let $i,j$ in $\S$. Then
\[\Tr{Jj}(T) = \phi(\Tr{Ji}(T)),\]
where $\phi$ is the isomorphism $\phi( z_0 + iz_1)= z_0 + jz_1$ between the complex fields $\C_i$ and $\C_j$.
\end{lemma}
\begin{proof}
Let $(e_n)_{n\in\N}$ be an orthonormal basis of $\hil_+^{Ji}$ and let $q\in\H$ with $|q| = 1$ such that $j = q^{-1} i q$, cf. \Cref{Turn}. Then $\phi(z) = q^{-1}zq$ for any $z\in\C_i$. Moreover, $(e_nq)_{n\in\N}$ is an orthonormal basis of $\hil_+^{Jj}$ as
\[ J(e_nq) = (Je_n)q = e_niq = (e_nq)q^{-1}iq = (e_nq) j. \]
Since $|q| = 1$, we have $q^{-1} = \overline{q}$ and thus, by \Cref{TCy}, we obtain
\begin{align*}
\phi(\Tr{Ji}(T)) = q^{-1}\Tr{Ji}(T)q = \sum_{n=0}^{+\infty} \langle e_nq, Te_nq\rangle = \Tr{Jj}(T).
\end{align*}

\end{proof}
\begin{lemma}\label{TraceProp}
Suppose $1 \leq p < +\infty$ and $\holder$. If $T \in S_p(J)$ and $S \in S_q(J)$, then
\begin{enumerate}[(i)]
\item $TS$ and $ST$ belong to the trace class $S_1(J)$
\item $\Tr{Ji}(TS) = \Tr{Ji}(ST)$
\item \label{TraceHoelder}$|\Tr{Ji}(TS)| \leq \|T\|_p \|S\|_q$.
\end{enumerate}
\begin{proof}
Applying \Cref{SchattenIso}, we can reduce the statement to the case of operators on a complex Hilbert space, where we know that these results are true. Hence,
\begin{align*}
&T \in S_p(J)~\text{and}~S\in S_q(J)\\
 \iff & \res{Ji}(T) \in S_p(\hil_+^{Ji})~\text{and}~\res{Ji}(S)\in S_q(\hil_+^{Ji})  \\
\implies & \res{Ji}(TS) \in S_1(\hil_+^{Ji})\\
\iff & TS \in S_1(J).
\end{align*}
 The second statement follows  from the definition:
\begin{align*}
\Tr{Ji}(TS) &= \tr\left(\res{Ji}(TS)\right)\\
&= \tr\left(\res{Ji}(T)\res{Ji}(S)\right)\\
&= \tr\left(\res{Ji}(S)\res{Ji}(T)\right)\\
&= \tr\left(\res{Ji}(ST)\right)\\
&= \Tr{Ji}(ST).
\end{align*}
The final statement relies on the fact that $\res{Ji}$ and $\lift{Ji}$ are p-norm preserving:
\begin{align*}
|\Tr{Ji}(TS)| &= |\tr\left(\res{Ji}(TS)\right)|\\
&= |\tr\left(\res{Ji}(T)\res{Ji}(S)\right)|\\
&\leq \|\res{Ji}(T)\|_p \|\res{Ji}(S)\|_q\\
&= \|T\|_p \|S\|_q,
\end{align*}
which finishes the proof of the lemma.

\end{proof}
\end{lemma}

\begin{lemma}
Suppose $1 \leq p <+ \infty$ and $\holder$. If $T \in S_p(J)$, then
\begin{align*}
\|T\|_p = \sup\left\{ |\tr_{Ji}(ST)|: \|S\|_q=1, S \in S_q(J)\right\}.
\end{align*}
\end{lemma}
\begin{proof}
Using the fact that this result is true for operators on a complex Hilbert space and that $\res{Ji}$ is a $p$-norm preserving isomorphism, we find that
\begin{align*}
\|T\|_p &= \|\res{Ji}(T)\|_p\\
&=  \sup\left\{ |\operatorname{tr}\left((S_\C\res{Ji}(T)\right)|: \|S_\C\|_q=1, S_\C \in S_q(\hil_+^{Ji})\right\}\\
&= \sup\left\{ |\operatorname{tr}\left((\res{Ji}(ST)\right)|: \|\res{Ji}(S)\|_q=1, S \in S_q(J)\right\}\\
&= \sup\left\{ |\operatorname{Tr}_{Ji}(ST)|: \|S\|_q=1, S \in S_q(J)\right\}.
\end{align*}
\end{proof}

We can also use the fact that $S_p(J)$ is isometrically isomorphic to the Schatten $p$-class of operators on the complex Hilbert space $\hil_+^{Ji}$ to determine its dual space.
\begin{theorem}\label{dual}
If $1 \leq p < +\infty$ and $\holder$, then
\begin{align*}
S_p(J)^* = S_q(J)
\end{align*}
with equal norms and under the pairing $\inpr{T}{S} = \operatorname{Tr}_{Ji}(TS).$
\end{theorem}
\begin{proof}
First, assume that $\xi$ is a continuous linear functional on $S_p(J)$.
By \Cref{SchattenIso}, the mapping $T_{\C}\mapsto \xi(\lift{Ji}(T_\C))$ is a linear
functional on $S_p(\hil_+^{Ji})$ with $\|\xi\circ\lift{Ji}\| = \|\xi\|$.

Since $S_p(\hil_+^{Ji})^* \cong S_q(\hil_+^{Ji})$ under the pairing
$\langle T_\C, S_\C\rangle = \tr (T_\C S_\C)$, there exists an operator
 $S_\C\in S_q(\hil_+^{Ji})$ with $\|S_\C\|_q = \|\xi\circ\lift{Ji}\|$ such that
\begin{align*}
 \xi\circ\lift{Ji}(T_\C) = \tr(T_\C S_\C),\quad \forall T_\C\in\hil_+^{Ji}.
\end{align*}
We define now
$S_{\xi} := \lift{Ji}(S_\C)$.
Then
\begin{align*}
\|S_\xi\|_q = \|S_\C\|_q = \|\xi\circ\lift{Ji}\| = \|\xi\|
\end{align*}
and
\[
\begin{split}
\xi(T) &= \xi\circ\lift{Ji}\circ\res{Ji}(T)
\\
&
= \tr(\res{Ji}(T)S_\C)
\\
&
= \tr(\res{Ji}(TS_\xi))
\\
&
= \Tr{Ji}(TS_\xi).
\end{split}
\]
Hence, the mapping $\Phi:\xi\mapsto S_\xi$ is an isometric $\C_i$-linear mapping of $S_p(J)^*$ into $S_q(J)$.

Conversely, it follows from the $\C_i$-linearity of the $Ji$-trace and (\ref{TraceHoelder}) of \Cref{TraceProp} that, for $S\in S_q(J)$, the mapping $T\mapsto \xi_S(T) := \Tr{Ji}(TS)$ is a bounded $\C_i$-linear functional on $S_p(J)$ with $\|\xi_S\| \leq \|S\|_q$. Consequently, $\Phi$ is even invertible and in turn $S_p(J)^*$ is isometrically isomorphic to $S_q(J)$.

\end{proof}

\subsection{Characterizations of Schatten class operators}

In this section we take a closer look at the singular values of Schatten class operators in order to arrive at necessary and sufficient conditions for an operator to belong to the Schatten class. A crucial observation was made in \Cref{SchattenIso}:
\[T\in S_p(J) \Longleftrightarrow \res{Ji}(T)\in S_p(\hil^{Ji}_+).\]

\begin{lemma}\label{SpS1Pos}
Let $T$ be a positive and compact operator on $\hil$ and $p\in(0,+\infty)$. Then
\begin{align*}
T \in S_p \iff T^p \in S_1.
\end{align*}
Moreover, $\|T\|_p^p = \|T_{}^p\|^{}_1$.
\end{lemma}
\begin{proof} Let $T = \sum_{n\in\N} e_n\langle e_n, \cdot \rangle \lambda_n$ be a spectral decomposition of $T$. Since the eigenvalues $\lambda_n$ are positive, they coincide with the singular values of $T$. Moreover, as in the case of complex operators, we have $\sigma_S(T) = \{ \lambda_n:n\in\N\}\cup\{0\}$, cf. \cite{COp}. The $S$-spectral theorem implies that $\sigma_S(T^p) = \{ \lambda_n^p:n\in\N\}\cup\{0\}$. Consequently, the singular value sequence of $T^p$ is $(\lambda_n^p)_{n\in\N}$, and hence
 \[ \|T^p\|_1 = \sum_{n\in\N}\lambda_n^p = \|T\|_p^p.\]

\end{proof}

\begin{theorem}\label{SpS1}
If $T$ is a compact operator on $\hil$ such that $[T,J]=0$ and $p\in( 0,+\infty)$ then
\begin{align*}
T \in S_p(J) \iff |T|^p=(T^*T)^{\frac{p}{2}} \in S_1(J) \iff T^*T \in S_{\frac{p}{2}}(J).
\end{align*}
Moreover
\begin{align*}
\|T\|_p^p = \| |T| \|_p^p = \| |T|^p \|_1 = \|T^*T\|^{\frac{p}{2}}_{\frac{p}{2}}.
\end{align*}
As a consequence, we have that
\begin{align*}
T \in S_p(J) \iff |T| \in S_p(J).
\end{align*}
\end{theorem}
\begin{proof}
We start by proving the first equivalence (the other ones follow analogously from the corresponding results for complex linear operators, cf.~\cite[Theorem~1.26]{kehe}):
\begin{align*}
T \in S_p(J) &\iff \res{Ji}(T) \in S_p(\hil_+^{Ji})\\
&\iff |\res{Ji}(T)|^p = \res{Ji}(|T|^p) \in S_1(\hil_+^{Ji})\\
&\iff |T|^p \in S_p(J).
\end{align*}
For the equality of the norms we prove the first one (the rest is proven in a similar way):
\begin{align*}
\|T\|_p^p &= \|\res{Ji}(T)\|_p^p= \||\res{Ji}(T)|\|_p^p= \|\res{Ji}(|T|)\|_p^p= \||T|\|_p^p.
\end{align*}

\end{proof}

Since $\res{Ji}$ and $\lift{Ji}$ are $p$-norm preserving, we easily obtain further characterizations of $(J,p)$-Schatten class operators from the respective results in the complex case, cf. \cite[Theorems~1.27--1.29]{kehe}.

\begin{theorem}
Suppose that $T$ is a compact operator on a quaternionic Hilbert-space $\hil$ with $[T,J]=0$ and that $p \geq 1$. Then $T$ is in $S_p(J)$ if and only if
\begin{align*}
\sum_{n\in\N} |\inpr{e_n}{Te_n}|^p < +\infty
\end{align*}
for all orthonormal sets $(e_n)_{n\in\N}$ in $\hil_+^{Ji}$. If $T$ is also selfadjoint then
\begin{align*}
\|T\|_p &= \sup \left\{ \left[\sum_{n \in \N}|\inpr{e_n}{Te_n} |^p\right]^{\frac{1}{p}}
\ : \  \{e_n\} ~\text{orthonormal set in}~ \hil_+^{Ji}\right\}.
\end{align*}
\end{theorem}

\begin{theorem}
Suppose that $T$ is a compact operator on $\hil$ with $[T,J]=0$ and that ${p \in[1,+\infty)}$. Then $T$ is in $S_p(J)$ if and only if
\begin{align*}
\sum_{n\in\N} |\inpr{\sigma_n}{Te_n}|^p < +\infty
\end{align*}
for all orthonormal sets $\{e_n\}$ and $\{\sigma_n\}$ in $\hil_+^{Ji}$. If $T$ is also positive then
\begin{align*}
\|T\|_p &= \sup \left\{ \left[\sum_{n\in\N}| \inpr{\sigma_n}{Te_n} |^p\right]^{\frac{1}{p}} : \{e_n\}~\text{and}~\{\sigma_n\}~\text{orthonormal sets in}~ \hil_+^{Ji}\right\}.
\end{align*}
\end{theorem}

\begin{theorem}
Suppose that $T$ is a compact operator on $\hil$ with $[T,J]=0$ and that ${0 < p \leq 2}$. Then, for any orthonormal basis $\{e_n\}$ of $\hil_+^{Ji}$, we have
\begin{align*}
\|T\|_p^p \leq \sum _{n = 1}^\infty \sum _{k=1}^\infty |\inpr{e_k}{Te_n}|^p.
\end{align*}
\end{theorem}

\begin{proposition}\label{SpCrit3}
Suppose $T$ is a positive, compact operator on $\hil$ and $x$ is a unit vector in $\hil$. Then
\begin{itemize}
\item $\inpr{x}{T^px} \geq \inpr{x}{Tx}^p$ for all $p \in [1,+\infty)$.
\item $\inpr{x}{T^px} \leq \inpr{x}{Tx}^p$ for all $0 < p \leq 1$.
\end{itemize}
\end{proposition}
\begin{proof}
Let $Tx = \sum_{n\in\N}  e_n \lambda_n \inpr{e_n}{x}$ be a spectral decomposition for the operator $T$. Then, for all $p >0$ and $x \in \hil$, we have by the spectral theorem that
\begin{align*}
T^px &= \sum_{n\in\N}  e_n \lambda_n^p \inpr{e_n}{x}
\end{align*}
and thus
\begin{align*}
\inpr{x}{T^px} &= \sum_{n\in\N} \lambda_n^p |\inpr{e_n}{x}|^2.
\end{align*}
We also know that for every Hilbert basis $(\xi_n)_{n\in\N}$ of $\hil$ we have that
\begin{align*}
\|x\|^2 = \sum _{n\in\N} |\inpr{\xi_n}{x}|^2,  \quad\forall x \in \hil.
\end{align*}
Let us first assume that $p \geq 1$ and let $q$ be its conjugate index. Applying H\"older's inequality  gives
\begin{align*}
\inpr{x}{Tx} &= \left(\sum _{n\in\N}  \lambda_n |\inpr{e_n}{x}|^2\right)\\
&= \left(\sum _{n\in\N}  \lambda_n |\inpr{e_n}{x}|^{\frac{2}{p}} |\inpr{e_n}{x}|^{\frac{2}{q}}\right)\\
&\leq \left(\sum _{n\in\N} \lambda_n^p |\inpr{e_n}{x}|^2\right)^{\frac{1}{p}}\left(\sum _{n\in\N} |\inpr{e_n}{x}|^2\right)^{\frac{1}{q}}\\
&= \inpr{x}{T^px}^{\frac{1}{p}}
\end{align*}
and since $p \geq 1$ taking the $p$-th power preserves the inequality.

If $0 < p \leq 1$ then we can find  $q \geq 1$ such that: $p+\frac{1}{q}=1$. Using the H\"older inequality with the conjugate pair $(\frac{1}{p},q)$ gives us that
\begin{align*}
\inpr{x}{T^px} &= \sum _{n\in\N} \lambda_n^p |\inpr{e_n}{x}|^2\\
&= \sum _{n\in\N} \lambda_n^p |\inpr{e_n}{x}|^{2p}|\inpr{e_n}{x}|^{\frac{2}{q}}\\
&\leq  \left(\sum _{n\in\N}\lambda_n |\inpr{e_n}{x}|^2\right)^p \left(\sum _{n\in\N} |\inpr{e_n}{x}|^2\right)^{\frac{1}{q}}\\
&= \inpr{x}{Tx}^p
\end{align*}
and this finishes the proof.

\end{proof}

\begin{corollary}
Suppose that $T$ is a positive, compact operator on $\hil$ such that $[T,J]=0$ and that $(e_n)_{n\in\N}$ is an orthonormal basis of $\hil$. If $p\in [1,+\infty)$, then the condition
\begin{align*}
\sum_{n\in\N} \inpr{e_n}{Te_n}^p < +\infty
\end{align*}
is necessary for $T \in S_p(J)$. If $0<p\leq 1$ then this condition is sufficient.
\end{corollary}
\begin{proof}
From \Cref{SpS1Pos} and \Cref{TracePos}, we have
\begin{align*}
T \in S_p(J) \iff T^p \in S_1(J) \iff \sum_{n\in\N} \inpr{e_n}{T^pe_n} < +\infty.
\end{align*}
Applying \Cref{SpCrit3} gives the result.

\end{proof}

\begin{theorem}
Suppose $T$ is a compact operator on $\hil$ with $[T,J]=0$ and $p\geq 2$, then
\begin{align*}
T \in S_p(J) \iff \sum_{n\in\N}\|Te_n\|^p <+ \infty
\end{align*}
for all orthonormal sets $\{e_n\}$ in $\hil_+^{Ji}$. Moreover,
\begin{align*}
\|T\|_p &= \sup \left\{ \left[\sum_{n\in\N}\|Te_n\|^p\right]^{\frac{1}{p}} \mid \{e_n\}~\text{orthonormal in}~\hil_+^{Ji}\right\}.
\end{align*}
\end{theorem}
\begin{proof}
Since $\res{Ji}$ is $p$-norm preserving, this follows immediately from the corresponding results for complex linear operators \cite[Theorem~1.33]{kehe}.
\end{proof}

\section{The Berezin Transform}
Let $\hil$ be a quaternionic reproducing kernel Hilbert space of functions on the unit ball, that is a Hilbert space of left slice hyperholomorphic functions in the unit ball $\D \subset \mathbb{H}$ with the property that for each $w \in \D$ the point evaluation ${f \mapsto f(w)}$ is a bounded right linear functional on $\hil$. From the Riesz representation theorem we know that there exists a unique function $K_w \in \hil$ such that
\begin{align*}
f(w)  = \inpr{K_w}{f}, & \qquad \forall f \in \hil.
\end{align*}
The function
\begin{align*}
K(q,w) := K_w(q), & \qquad q,w \in \D
\end{align*}
is called the reproducing kernel of $\hil$.
\begin{lemma}\label{Dense1}
The functions $\{K_w \mid w \in \D\}$ span the entire space $\hil$.
\end{lemma}
\begin{proof}
This follows immediately from the reproducing property. If $f \perp K_w$ for all $w\in\D$, then
\begin{align*}
f(w) = \inpr{K_w}{f} = 0
\end{align*}
and thus $f=0$.

\end{proof}
Since we consider slice hyperholomorphic functions, we can use their specific structure to prove a stronger result.
\begin{lemma}\label{Dense2}
Let $i \in \S$ and set $\D_i := \D \cap \C_i$. The set $\{K_\omega : \omega \in \D_i\}$ spans the space $\hil$.
\end{lemma}
\begin{proof}
If $f\perp K_w$ for any $w\in \D_i$, then $f(w) = \inpr{K_w}{f} = 0$ for all $w\in \D_i$. The representation formula, \Cref{RepFo}, then implies $f \equiv 0 $, and hence $\hil = \overline{\linspan\{K_w: w\in\D_i\}}$.

\end{proof}

The kernel $K(q,w)$ is obviously left slice hyperholomorphic in~$q$. Furthermore, it is right slice hyperholomorphic in  $\overline{w}$, i.e. the mapping $w \mapsto K(q,\overline{w})$ is right slice hyperholomorphic. This follows from \Cref{IDeriv} because
\[
\begin{split}
K(q,\overline{w})\rbpartial{iw} &= \frac12 \langle K_q, K_{\overline{w}}\rangle (\partial_{w_0}+i_{w}\partial_{w_1})
\\
&
= \frac12 \overline{(\partial_{w_0}-i_{w}\partial_{w_1})\langle K_{\overline{w}},K_q\rangle}
\\
&
= \frac12\overline{ (\partial_{w_0}+i_{\overline{w}}\partial_{w_1}) K_q(\overline{w})} = 0.
\end{split}
\]
Pointing out that
\begin{align*}
\|K_q\|^2 = K(q,q) \qquad\text{and}\qquad |K(q,w)|^2\leq K(q,q)K(w,w),
\end{align*}
we can conclude that, if for each $q \in \D$ there exists $f \in \hil$ such that $f(q) \neq 0$, then
\begin{align*}
K(q,q)>0 \quad\forall q \in \D.
\end{align*}
We shall assume this to be true in the following. We can then normalize the reproducing kernels to obtain a family of unit vectors $k_q$ by
\begin{align*}
k_q(w) = \frac{K(w,q)}{\sqrt{K(q,q)}}\qquad\text{for } w \in \D.
\end{align*}
We call these the normalized reproducing kernels of $\hil$.

For the following discussion, we fix an imaginary unit $\bi\in\S$. Furthermore, we assume that there exists a unitary and antiselfadjoint operator $J$ such that $\hil_{+}^{J\bi} = \overline{\linspan_{\C_{\bi}}\{K_q, q\in\D_{\bi}\}}$.
\begin{definition}
Let $T$ be a bounded linear operator on $\hil$ such that $[T,J]=0$. The function
\begin{align*}
\widetilde{T}(q) = \inpr{k_q}{Tk_q}, \quad q\in\D_{\bi}
\end{align*}
is called the Berezin transform of $T$.
\end{definition}

\begin{proposition}
The Berezin transform has the following properties:
\begin{enumerate}[(i)]
\item If $T$ is self-adjoint, then $\widetilde{T}$ is real-valued.
\item If $T$ is positive, then $\widetilde{T}$ is non-negative.
\item We have $\widetilde{T^*} = \overline{\widetilde{T}}$.
\item The mapping $T\mapsto \widetilde{T}$ is a contractive $\C_{\bi}$-linear mapping from $\B_J(\hil)$ into $L^{\infty}(\D_{\bi},\H)$.
\end{enumerate}
\end{proposition}
\begin{proof}
For $q\in\D_{\bi}$, it is
\[ \widetilde{T^*}(q) = \langle k_q,T^* k_q\rangle = \langle Tk_q, k_q\rangle = \overline{\langle k_q, Tk_q\rangle} = \overline{\widetilde{T}(q)}\]
and hence (iii) holds. If $T$ is self-adjoint, this implies $\overline{\widetilde{T}}(q) = \widetilde{T}(q)$ and so $\widetilde{T}(q)$ is real. For positive $T$, the definition of positivity immediately implies $(ii)$.

 The Berezin transform is obviously $\R$-linear. Since $T\in\boundOP_J(\hil)$, it maps $k_q\in\hil_+^{J\bi}$ to an element in $\hil_+^{J\bi}$, and hence its Berezin transform $\widetilde{T}$ takes values in $\C_{\bi}$. For $\bi T = JT$, we have
 \[ \widetilde{\bi T}(q) = \langle k_q, JT k_q\rangle = \langle k_q, T k_q \bi\rangle  =\widetilde{T}(q)\bi = \bi \widetilde{T}(q).  \]
Thus, the Berezin transform is even $\C_{\bi}$-linear.

 Finally, we deduce from the Cauchy-Schwarz-inequality that
\begin{align*}
\widetilde{T}(q) = \langle k_q, Tk_q\rangle \leq \|k_q\|\|Tk_q\| \leq \|T\| \|k_q\|^2 = \|T\|.
\end{align*}
Hence, $\widetilde{T}\in L^{\infty}(\D)$ with $\|\widetilde{T}\|_{\infty}\leq \|T\|$.

\end{proof}

Observe that the Berezin Transform of $T\in\boundOP_J(\hil)$ coincides with the classical Berezin transform of the operator $\res{Ji}(T)\in\boundOP_J(\hil)$. Since the restriction operator $\res{Ji}$ and the classical Berezin transform are injective (cf. \cite[Proposition~6.2.]{kehe}), we immediately obtain the following Lemma.
\begin{lemma}
The Berezin-transform is one-to-one.
\end{lemma}

\section{The Case of Weighted Bergman Spaces}

In this section we consider the special case of weighted Bergman spaces on the unit ball. A first study of these spaces in the slice hyperholomorphic setting has been done in \cite{BBD}. We recall the main definitions and results.
\begin{definition}
Let $i\in\S$ and let $dm_i$ be the Lebesgue measure on the complex plane $\C_i$. For $\alpha >-1$, we define the measure $dA_{\alpha,i}(z)$ on the unit ball $\D_i := \D\cap\C_i$ in $\C_i$ by
\[
dA_{\alpha,i}(z) = \frac{\alpha+1}{\pi}(1-|z|^2)^{\alpha}\,dm_i(z).
\]
For $p>0$, the weighted slice Bergman space $\berg_{\alpha,i}^p(\D)$ is the quaternionic right vector space of all left slice hyperholomorphic functions $f$ on $\D$ such that
\[\int_{\D_i} |f(z)|^p\,dA_{\alpha,i}(z) < + \infty.\]
For $f\in \berg_{\alpha,i}^p(\D)$, we define
\[ \|f\| _{p,\alpha,i} := \left(\int_{\D_i}|f(z)|^p\,dA_{\alpha,i}(z)\right)^{\frac{1}{p}}.\]
\end{definition}

\begin{corollary}\label{BergSplit}
Let $i,j\in\S$ with $i\perp j$, let $f\in\lhol(\D)$ and write $f_i = f_1 + f_2j$ with holomorphic functions $f_1,f_2:\D_i\to\C_i$, cf. \Cref{SplitLem}. Then $f\in\berg_{\alpha,i}^p(\D)$ if and only if $f_1$ and $f_2$ belong to the complex Bergman space $\berg_{\C,\alpha}^{p}(\D)$, i.e. the space of all holomorphic functions $g$ on $\D_i$ such that $\int_{\D_i} |g(z)|^p\, dA_{\alpha,i}(z)<+\infty$.
\end{corollary}
\begin{lemma}\label{BergIJ}
Let $\alpha >-1$, $p>0$ and $i,j\in\S$. A left slice hyperholomorphic function $f$ belongs to $\berg_{\alpha,i}^p(\D)$ if and only if it belongs to $\berg_{\alpha,j}^p(\D)$. Moreover,
\[ \|f\|^p_{p,\alpha,i} \leq 2^{\max\{p,1\}}  \|f\|^p_{p,\alpha,j} \leq  2^{2 \max\{p,1\}} \|f\|^p_{p,\alpha,i}.\]
\end{lemma}

\begin{definition}
For $\alpha>-1$ and $p>0$, we define the weighted slice hyperholomorphic Bergman space $\berg_{\alpha
}^p(\D)$ as the space of all left slice hyperholomorphic functions $f$ such that
\[\|f\|_{p,\alpha} := \sup_{i\in\S}\| f\|_{p,\alpha,i}<+\infty.\]
\end{definition}
\begin{remark}
Observe that \Cref{BergSplit} implies that, for each $i\in\S$, the spaces $\berg_{\alpha}^p(\D)$ and $\berg_{\alpha, i}^p(\D)$ contain the same elements and that their norms are equivalent.
\end{remark}

\subsection{Further properties of the weighted slice Bergman space $\berg_{\alpha,i}^2(\D)$}
As in the complex space, the norm on the slice Bergman space $\berg_{\alpha,i}^2(\D)$ is generated by the scalar product
\begin{equation}\label{BergScal}
\langle f,g\rangle_{2,\alpha,i} := \int_{\D_i} \overline{f(z)}g(z)\, dA_{\alpha,i}(z).
\end{equation}
This is however not true for the slice hyperholomorphic Bergman space $\berg_{\alpha}^2(\D)$: orthogonality and other concepts related to the scalar product will always depend on the complex plane chosen to define the scalar product.  Nevertheless, by \Cref{BergIJ}, independently of the choice of the complex plane, the norm topology on $\berg_{\alpha}^p(\D)$ is generated by the chosen scalar product.

\begin{lemma} Endowed with the scalar product defined in \eqref{BergScal}, the slice Bergman space $\berg_{\alpha,i}^2(\D)$ turns into a  reproducing kernel quaternionic Hilbert space.
\end{lemma}
\begin{proof}
If $f\in\berg_{\alpha,i}^2(\D)$ and $w\in \D$, then we can chose $j\in\S$ with $j\perp i_w$ and write the restriction of $f$ to the plane $\C_{i_w}$ as $f_{i_w} = f_1 + f_2 j$ with holomorphic functions $f_1,f_2: \D_{i_w}\to \C_{i_w}$. By \Cref{BergSplit}, the functions $f_1$ and $f_2$ belong to the complex Bergman space $\berg_{\C,\alpha}^2(\D_i)$. Since point evaluations are continuous functionals on $\berg_{\C,\alpha}^2(\D_i)$, cf. \cite[Theorem~4.14]{kehe}, we have
\begin{gather*}
|f(w)| \leq |f_1(w)| + |f_2(w)| \\
\leq C(\| f_1\|_{\C,2,\alpha} + \| f_2\|_{\C,2,\alpha})
\\
 \leq 2C \|f\|_{2,\alpha,i_{\omega}}
  \leq \widetilde{C} \|f\|_{2,\alpha,i},
\end{gather*}
where the last equality follows from the equivalence of the Bergman slice norms, cf. \Cref{BergIJ}. Hence, point evaluations are continuous linear functionals on~$\berg_{\alpha,i}^2(\D)$.

\end{proof}

\begin{definition}
Let $\alpha > -1$. We define the slice hyperholomorphic $\alpha$-Bergman kernel for $q,w\in\D$ as
\[
K_{\alpha}(q,w):=\frac12(1-i_qi)\frac{1}{(1 - q_{i_w}\overline{w})^{2+\alpha}} + \frac12(1+i_qi)\frac{1}{(1-\overline{q_{i_w}}\,\overline{w})^{2+\alpha}}.
\]
\end{definition}
\begin{remark}\label{BKerRes}
Observe that $K_{\alpha}(q,w)$ is  an extension of the complex Bergman kernel
\[K_{\C,\alpha}(z,w) = \frac{1}{(1-z\overline{w})^{2+\alpha}}.\]
Whenever $q$ and $w$ belong to the same complex plane, $K_{\alpha}(q,w) = K_{\C,\alpha}(q,w)$.
  Moreover, by a direct computation it is easy to see that
  the kernel $K_{\alpha}(x,w)$ is left slice hyperholomophic in $q$ and right slice hyperholomorphic in $\overline{w}$.
\end{remark}

\begin{lemma}
The function $K_{\alpha}(\cdot,\cdot)$ is the reproducing kernel of $\berg_{\alpha,i}^2(\D)$.
\end{lemma}
\begin{proof} Let $f\in\berg_{\alpha,i}^2(\D)$ and set $K_w(q) := K_{\alpha}(q,w)$ for $w\in\D$. We show that $K_w\in\berg_{\alpha,i}^2(\D)$ and $\langle K_w ,f \rangle_{2,\alpha,i} = f(w)$.

First consider $w\in\D_i$. Since $K_w|_{\D_i}$ is nothing but the complex Bergman kernel $K_{\C,\alpha}(\cdot,\cdot)$, we immediately obtain from \Cref{BergSplit} that $K_w$ belongs to  $\berg_{\alpha,i}^2(\D)$ and that we can write $f_{i} = f_1 + f_2 j$ with $f_1,f_2\in\berg_{\C,\alpha}^2(\D_i)$. From \Cref{BKerRes} and the reproducing property of $K_{\C,\alpha}(\cdot,\cdot)$ in $\berg_{\C,\alpha}^2(\D)$, we obtain
\begin{align*}
\langle K_w, f\rangle_{2,\alpha,i} &= \int_{\D_i}\overline{K_{\alpha}(z,w)}f(z)\, d A_{\alpha,i}(z) \\
&= \int_{\D_i}\overline{K_{\C,\alpha}(z,w)}f_1(z)\, dA_{\alpha,i}(z) + \int_{\D_i}\overline{K_{\C,\alpha}(z,w)}f_{2}(z)\, d A_{\alpha,i}(z)j\\
&=  f_1(w) + f_2(w)j  = f(w).
\end{align*}
If $w\notin\D_i$, then \Cref{RepFo} implies
\[K_{w} = K_{w_i} (1-ii_w)\frac12  +  K_{\overline{w_i}}(1+ii_w)\frac12,\]
because of the right slice hyperholomorphicity of $K_{\alpha}( q,w)$ in $\overline{w}$. Hence, the function $K_w$ belongs to $\berg_{\alpha,i}^2(\D)$ because it is a right linear combination of the functions  $K_{w_i}$ and $K_{\overline{w_i}}$, which belong to $\berg_{\alpha,i}^2(\D)$ by the above argumentation. From the representation formula, we finally also deduce
\begin{align*}
\langle K_w, f\rangle_{\alpha,i} &= \frac12 \overline{(1-ii_w)}\langle K_{w_i},f\rangle_{\alpha,i} +  \frac12 \overline{(1+ii_w)}\langle K_{\overline{w_i}},f\rangle_{\alpha,i} &\\
&=\frac12 (1-i_wi) f(w_i) + \frac12 (1+i_wi)f(\overline{w_i}) = f(w).
 \end{align*}

\end{proof}

\subsection{The Berezin transform on $\berg_{\alpha,i}^2(\D)$}
In this section we consider the following fixed unitary and anti-selfadjoint operator
\[J f := \ext( i f_i), \quad \forall f \in \berg_{\alpha,i}^2(\D)\]
and consider it as a left scalar multiplication with $i\in\S$.

\begin{corollary}\label{JH+}
It is
\[ (\berg_{\alpha,i}^2(\D))_+^{Ji} = \overline{\linspan_{\C_i}(K_q: q\in\D_i)}\cong \berg_{\C,\alpha}^{2}(\D_i).\]
\end{corollary}
\begin{proof}
Write $f_i$ for $f\in\berg_{\alpha,i}^2(\D)$ as $f_i = f_1 + f_2j$ with components that are holomorphic on $\D_i$. By the right linearity of the extension operator, we then have $f = \ext(f_1) + \ext(f_2)j$. From
\[
\begin{split}
Jf &= \ext(if_i)
\\
&
= \ext(f_1i) + \ext(f_2ij)
\\
&
 = \ext(f_1)i - \ext(f_2)ji,
\end{split}
\]
we deduce that $f\in(\berg_{\alpha,i}^2(\D))_+^{Ji}$ if and only if $f = \ext(f_1)$, i.e. if and only if $f_2 = 0$. The mapping $\varphi: f\mapsto f_i$ is therefore an isometric isomorphism between $(\berg_{\alpha,i}^2(\D))_+^{Ji}$ and $\berg_{\C,\alpha}^2(\D_i)$ and so $(\berg_{\alpha,i}^2(\D))_+^{Ji}\cong \berg_{\C,\alpha}^2(\D_i)$. Since $\varphi(K_q) = K_{\C,q}$ and $\linspan_{\C_i}\{K_{\C,q}, q\in\D_i\}$ is dense in $\berg_{\C,\alpha}^2(\D_i)$, we obtain that
\[
\begin{split}
\berg_{\alpha,i}^2(\D) &
= \varphi^{-1}\left(\berg_{\C,\alpha}^2(\D_i)\right)
\\
&
= \varphi^{-1}\left(\overline{\linspan_{\C_i}\{K_{\C,q}, q\in\D_i\}}\right)
\\
&
= \overline{\linspan_{\C_i}\{\varphi^{-1}(K_{\C,q}), q\in\D_i\}}
\\
&
=  \overline{\linspan_{\C_i}\{K_{q}, q\in\D_i\}}.
\end{split}
\]

\end{proof}

Recall that a bounded linear operator $T$ on $\berg_{\alpha,i}^2(\D)$ satisfies $[T,J]=0$ if and only if $T = \lift{Ji}(T_{\C})$ for some bounded linear operator on $T_{\C}$ on the complex Bergman space $\berg_{\C,\alpha}^2(\D_i)$. In this case $\widetilde{T}(z) = \widetilde{T_{\C}}(z)$, where $\widetilde{T_{\C}}$ denotes the Berezin transform of the complex operator $T_{\C}$.

\begin{theorem}\label{TraceInt}
Let $T$ be a bounded positive operator on $\berg_{\alpha,i}^2(\D)$ and let $(\lambda_{n})_{n\in\N}$ be its sequence of singular values. If
\[d\mu_i(z) = \frac{1}{\pi(1-|z|^2)^2} \,dm_i\]
 is the $\C_i$-M\"obius invariant area measure on $\D_i$, then
\[ \sum_{n=1}^{+\infty} \lambda_n = (\alpha+1)\int_{\D_i}\tilde{T}(z)\, d\mu_i(z).\]
\end{theorem}
\begin{proof}
Let $(e_n)_{n\in\N}$ be any orthonormal basis of $\berg_{\alpha,i}^2(\D)$. By \Cref{TracePos}, we have  $\sum_{n=1}^{+\infty}\lambda_n = \sum_{n=0}^{+\infty}\langle e_n,Te_n\rangle$.
Set $S = \sqrt{T}$. Fubini's theorem, the reproducing property of $K_z$ and Parseval's identity imply
\[
\begin{split}
\sum_{n=1}^{+\infty}\langle e_n, Te_n\rangle &= \sum_{n=1}^{+\infty} \|Se_n\|^2 = \sum_{n=1}^{+\infty}\int_{\D_i} |S e_n(z)|^2\,dA_{\alpha,i}(z)
\\
&= \int_{\D_i} \sum_{n=1}^{+\infty} |S e_n(z)|^2\,dA_{\alpha,i}(z) = \int_{\D_i} \sum_{n=1}^{+\infty} |\langle K_z, Se_n\rangle|^2\,dA_{\alpha,i}(z)
\\
&= \int_{\D_i} \sum_{n=1}^{+\infty} |\langle e_n, SK_z\rangle|^2\,dA_{\alpha,i}(z) =  \int_{\D_i}  \| SK_z\|^2\,dA_{\alpha,i}(z)
\\
&=  \int_{\D_i}  \langle K_z, TK_z \rangle\,dA_{\alpha,i}(z) = \int_{\D_i}\widetilde{T}(z)K(z,z)\,dA_{\alpha,i}(z)
\\
&= (\alpha + 1) \int_{\D_i}\widetilde{T}(z)\,d\mu_i(z).
\end{split}
\]
\end{proof}
We have the following important consequence:
\begin{corollary}\label{PosL1}
Let $J$ be any unitary anti-selfadjoint operator on $\berg_{\alpha,i}^2(\D)$. A  positive operator $T\in\boundOP_J(\berg_{\alpha,i}^2(\D))$ belongs to the trace class $S_1(J)$ if and only if $\widetilde{T}\in L^1(\D_i,d\mu_i)$.
\end{corollary}

\begin{corollary}\label{PosL2}
If $T\in S_1(J)$, then $\widetilde{T}$ is in $L^1(\D_i,d\mu_i)$ and
\[\Tr{Ji}(T) = (\alpha+1)\int_{\D_i}\widetilde{T}(z)\, d\mu_i(z). \]
\end{corollary}
\begin{proof}
Set $T_{\C} = \res{Ji}(T)$, write
\[ T_{\C} = T_{\C,1} - T_{\C,2} + i(T_{\C,2} - T_{\C,3})\]
with positive operators $T_{\C,\ell}\in \boundOP(\berg_{\C,\alpha}^2(\D_i))$ and set $T_\ell := \lift{Ji}(T_{\C,\ell})$ for $\ell = 1,\ldots,4$. Then the operators $T_\ell$ are positive and $T = T_1 - T_2 + J T_3 - J T_4$. By the $\C_i$-linearity of the $Ji$-trace, the $\C_i$-linearity of the Berezin transform and \Cref{TraceInt}, we have
\begin{align*}\Tr{Ji}(T) =& \Tr{Ji}(T_1) - \Tr{Ji}(T_2) + i \Tr{Ji}(T_3) - i \Tr{Ji}(T_4)\\
=& (\alpha+1)\int_{\D_i}\widetilde{T_1}(z)\, d\mu_i(z) - (\alpha+1)\int_{\D_i}\widetilde{T_2}(z)\, d\mu_i(z) \\
&+i(\alpha+1)\int_{\D_i}\widetilde{T_3}(z)\, d\mu_i(z) - i(\alpha+1)\int_{\D_i}\widetilde{T_4}(z)\, d\mu_i(z)\\
=& (\alpha+1)\int_{\D_i}\widetilde{T}(z)\, d\mu_i(z).
\end{align*}

\end{proof}

\begin{theorem}\label{AnotherTheorem}
Let $T$ be a positive operator on $\berg_{\alpha,i}^2(\D)$ such that $[T,J]  = 0$.
\begin{enumerate}[(i)]
\item If $1\leq p <+\infty$ and $T\in S_p(J)$, then $\widetilde{T} \in L^p(\D_i,d\mu_i)$.
\item If $0<p\leq 1$ and $\widetilde{T}\in L^p(\D_i,d\mu_i)$, then $T\in S_p(J)$.
\end{enumerate}
\end{theorem}
\begin{proof}
For $1\leq p < +\infty$, \Cref{SpCrit3} implies
$\widetilde{T^p}(z)\geq (\widetilde{T}(z))^p$ for $z\in\D_i$. Because $T\in S_p(J)$, we have by \Cref{SpS1Pos} that $T^p\in S_1(J)$, and hence, we deduce from \Cref{TraceInt}
\[\int_{\D_i} \left(\widetilde{T}(z)\right)^{p}\,d\mu_i(z) \leq \int_{\D_i}\widetilde{T^p}(z)\, d\mu_i(z) = \Tr{Ji}(T^p)<+\infty.\]
The second statement is proved in a similar way.

\end{proof}
\begin{corollary}
If $1\leq p < + \infty$ and $T\in S_p(J)$, then $\widetilde{T}\in L^p(\D_i, d\mu_i)$.
\end{corollary}
\begin{proof}
This follows  from \Cref{AnotherTheorem} and the fact that we can write $T$, in terms of the $\C_i$-Banach space structure defined on $S_p(J)$, as a $\C_i$-linear combination of positive operators. See the proof of \Cref{PosL2}.

\end{proof}

For $z\in\D_i$, we denote by $P_z$ the orthogonal projection of $\berg_{\alpha,i}^2(\D)$ onto the one-dimensional subspace generated by $k_z$, i.e.
\[P_z(f) =  k_z\langle k_z,f\rangle, \quad \forall f\in\berg_{\alpha,i}^2(\D).\]
Obviously, $P_z$ is a positive operator with rank one. Observe that $k_z\in\hil_{+}^{Ji}$ by \Cref{JH+} and hence $Jk_z = k_zi$, from which we deduce
\[
\begin{split}
JP_z f &= Jk_z\langle k_z, f\rangle = k_z i \langle k_z, f\rangle
\\
&
= k_z \langle -k_z i, f\rangle = k_z\langle J^*k_z,f\rangle
\\
&
 = k_z \langle k_z, Jf\rangle
 \\
 &
  = P_zJf
\end{split}
\]
for all $f\in\berg_{\alpha,i}^2(\D)$.
Hence, $[P_z,J] = 0$ and $P_z \in S_1(J)$ as $\Tr{Ji}( P_z) = 1$.

\begin{corollary}\label{cbaf}
Let $T\in\B_J(\berg_{\alpha,i}^2(\D))$. Then
\[\widetilde{T}(z) = \Tr{Ji}(TP_z), \quad \forall z\in\D_i.\]
\end{corollary}
\begin{proof}
Let $(e_n)_{n\in\N}$ be an orthonormal basis of $(\berg_{\alpha,i}^2(\D))_{+}^{Ji}$ with $e_1 = k_z$. Then
\[\Tr{Ji}(TP_z) = \sum_{n=1}^{+\infty}\langle e_n,TP_ze_n\rangle = \langle k_z,Tk_z\rangle = \widetilde{T}(z).\]

\end{proof}

\begin{corollary}
Let $T\in\B_J(\berg_{\alpha,i}^2(\D))$. Then
\[ |\widetilde{T}(z) - \widetilde{T}(w)| \leq \| T\| \|P_z - P_w\|_{1}, \quad \forall z\in\D_z.\]
\end{corollary}
\begin{proof}
By \Cref{cbaf}, we have
\[\widetilde{T}(z) - \widetilde{T}(w) = \Tr{Ji}(TP_z) - \Tr{Ji}(TP_w) = \Tr{Ji}(T(P_z-P_w)).\]
Hence, we deduce from \Cref{ideal} and (\ref{TraceHoelder}) in \Cref{TraceProp} that
\[| \widetilde{T}(z) - \widetilde{T}(w)| \leq |\Tr{Ji}(T(P_z-P_w)) | \leq \|T\| \|P_z-P_w\|_1.  \]

\end{proof}

\begin{lemma}
Let $z,w\in\D_i$. Then
\[ \|P_z - P_w\| = \left(1 - |\langle k_w, k_z\rangle|^2\right)^{\frac{1}{2}},\]
while
\[ \|P_z - P_w \|_1= 2 \left(1-|\langle k_w,k_z\rangle|^2\right)^{\frac{1}{2}}.\]
\end{lemma}
\begin{proof}
These equalities follow immediately from the corresponding equalities in the complex case, cf. \cite[ Lemma~6.10]{kehe}, and the following facts: $\res{J,i}$ preserves both the operator norm and the Schatten norm (cf. \Cref{BJIso} and \Cref{SchattenIso}) and $\langle k_w , k_z\rangle = \langle k_{w,i}, k_{z,i}\rangle_{\C}$, where $\langle\cdot,\cdot\rangle_{\C}$ denotes the scalar product of the complex Bergman space $\berg_{\C,i}^{2}(\D_i)$ and $k_{w,i}$ denotes the restriction of $k_w$ to $\D_i$.
Indeed, denoting $P_{z,\C}: = \res{Ji}(P_z)$ and $P_{w,\C} := \res{Ji}(P_w)$, we have
\begin{gather*}
\| P_z - P_w \| = \|P_{z,\C} - P_{w,\C}\| = \left(1 - |\langle k_{w,i}, k_{z,i}\rangle_{\C}|^2\right)^{\frac{1}{2}} = \left(1 - |\langle k_{w}, k_{z}\rangle|^2\right)^{\frac{1}{2}}.
\end{gather*}
The case of the Schatten-norm follows analogously.

\end{proof}
Finally, we obtain Lipschitz estimates for the Berezin transform analogue to those of the complex case.
\begin{theorem}
Let $T\in\B_J(\berg_{\alpha,i}^2(\D))$. Then
\[ |\widetilde{T}(z) - \widetilde{T}(w) | \leq 2 \sqrt{2+\alpha} \| T \| \rho(z,w), \quad \forall z,w\in\D_i,\]
where
\[ \rho(z,w) = \frac{|z-w|}{|1-z\overline{w}|}\]
is the pseudo-hyperbolic metric between $z$ and $w$. Furthermore, the Lipschitz constant $2\sqrt{2+\alpha}$ is sharp.
\end{theorem}
\begin{proof}
Again, we apply the corresponding result for complex linear operators \cite[Theorem~6.11]{kehe}. Denoting $T_{\C} = \res{Ji}(T)$, we have
\[
\begin{split}
|\widetilde{T}(z) - \widetilde{T}(w)| &= |\widetilde{T_{\C}}(z) - \widetilde{T_{\C}}(w)|
\\
&
\leq 2\sqrt{2+\alpha}\|T_{\C}\|\rho(z,w)
\\
&
= 2\sqrt{2+\alpha}\|T\|\rho(z,w).
\end{split}
\]
\end{proof}
\begin{theorem}
Let $T\in\B_J(\berg_{\alpha,i}^2(\D))$. Then
\[ |\widetilde{T}(z) - \widetilde{T}(w) | \leq 2 \sqrt{2+\alpha} \| T \| \beta(z,w), \quad \forall z,w\in\D_i,\]
where
\[ \beta(z,w) = \frac{1}{2}\log\frac{1+\rho(z,w)}{1-\rho(z,w)},\]
is the Bergman metric on $\D_i$. Furthermore, the Lipschitz constant $2\sqrt{2+\alpha}$ is sharp.
\end{theorem}
\begin{proof}
Once more, we deduce this from the corresponding result for complex linear operators \cite[Theorem~6.11]{kehe}. Denoting $T_{\C} = \res{Ji}(T)$, we have
\[
\begin{split}
|\widetilde{T}(z) - \widetilde{T}(w)| &= |\widetilde{T_{\C}}(z) - \widetilde{T_{\C}}(w)|
 \\
 &
\leq 2\sqrt{2+\alpha}\|T_{\C}\|\beta(z,w)
\\
&
= 2\sqrt{2+\alpha}\|T\|\beta(z,w).
\end{split}
\]
and this concludes the proof.
\end{proof}
We conclude this paper with a remark on further applications of slice hyperholomorphicity and quaternionic operators.
\begin{remark}
Classical Schur analysis is an important branch of operators theory with several applications in science and in technology, see for example the book \cite{MR2002b:47144},
 the notion of $S$-spectrum and of $S$-resolvent operators appear in Schur analysis in the slice hyperholomorphic setting
in the realization of Schur functions, see the foundational paper \cite{acs1}.
\\
The literature on Schur analysis in the slice hyperholomorphic setting is nowadays very well developed we mention just some of the
main results that are contained in the papers \cite{acs1,acs2,acs3,milano} and in the book \cite{ACSBOOK}.
The main reference for slice hyperholomorphic functions are the books \cite{ACSBOOK,SCalcBook,GSSb}.
\end{remark}

\end{document}